%% file: main.tex
\documentclass[english,12pt,oneside]{amsproc}
\usepackage[english]{babel}
\usepackage{a4wide}
\usepackage{amsthm}
\usepackage{graphics}
\usepackage{amsfonts, amssymb, amscd, amsmath}
\usepackage{mathdots}
\usepackage{latexsym}
\usepackage[matrix,arrow,curve]{xy}
\usepackage{mathabx}
\usepackage{color}
\usepackage{pbox}
\usepackage{tikz}
\usetikzlibrary{matrix,decorations.pathreplacing,positioning}
\usepackage{hyperref}
\usepackage{subcaption}
\usepackage{enumitem}
\usepackage{ifthen}
\usepackage{amsrefs}
\DeclareMathOperator{\Hom}{Hom}
\DeclareMathOperator{\rk}{rk}

\DeclareMathOperator{\relint}{relint}

\DeclareMathOperator{\str}{star}

\DeclareMathOperator{\Cube}{Cube}
\DeclareMathOperator{\Aut}{Aut}
\newcounter{stmcounter}[section]

\numberwithin{equation}{section}

\newcounter{thmMaincounter}

\theoremstyle{plain}
\newtheorem{cor}[stmcounter]{Corollary}

\newtheorem{thm}[stmcounter]{Theorem}
\newtheorem{thmM}[thmMaincounter]{Theorem}

\newtheorem{prop}[stmcounter]{Proposition}
\newtheorem{lem}[stmcounter]{Lemma}

\theoremstyle{definition}
\newtheorem{defin}[stmcounter]{Definition}

\theoremstyle{remark}
\newtheorem{ex}[stmcounter]{Example}
\newtheorem{rem}[stmcounter]{Remark}
\newtheorem{con}[stmcounter]{Construction}

\newcommand{\la}{\langle}
\newcommand{\ra}{\rangle}
\newcommand{\lb}{\lbrace}
\newcommand{\rb}{\rbrace}
\newcommand{\floor}[1]{\lfloor #1 \rfloor}
\def\Co{\mathbb C}
\def\Ro{\mathbb R}
\def\Qo{\mathbb Q}
\def\Zo{\mathbb Z}

\def\No{\mathbb N}
\def\Io{\mathbb I}
\begin{document}
\title[Independent GKM-graphs]{On independent GKM-graphs without nontrivial extensions}

\author{Grigory Solomadin}
\address{Laboratory of algebraic topology and its applications, Faculty of computer science, National Research University Higher School of Economics, Russian Federation}
\email{grigory.solomadin@gmail.com}

\thanks{The research was partially supported by the ``RUDN University program 5--100'' program. The article was prepared within the framework of the HSE University Basic Research Program}


% \subjclass[2020]{Primary: 57S12, 55N91, 06A06 Secondary: 55U10}

\begin{abstract}
In this paper an example of a $k$-independent $(n,k)$-type GKM-graph without nontrivial extensions is constructed for any $n\geq k\geq 3$. It is shown that this example cannot be realized by a GKM-manifold for any $n=k=3$ or $n\geq k\geq 4$.
\end{abstract}

\maketitle

\section{Introduction}\label{secIntro}

GKM-theory \cite{gkm-98}, \cite{gu-za-01} provides with a useful method for computation of the equivariant cohomology ring for a wide class of manifolds equipped with a torus action (having only isolated fixed points) called GKM-manifolds. The main tool of this theory is a graph that is naturally asssociated with the orbit space of the equivariant $1$-skeleton for a GKM-manifold. The labels on the edges of this graph are given by the tangent weights of the fixed points associated with the torus action. By abstracting from the torus action one obtains the definition of a GKM-graph axiomatized in \cite{gu-za-01}. Apart from the equivariant cohomology ring many other important geometric objects and topological invariants related to a GKM-manifold can be studied in terms of the respective GKM-graph, such as Betti numbers \cite{gu-za-01}, invariant almost complex structures \cite{gu-ho-za-06}, invariant symplectic and K\"ahler structures \cite{go-ko-zo-20}, etc.

A GKM-manifold $M^{2n}$ is called a GKM-manifold in $j$-general position if any $j$ tangent weights at any fixed point of this action on $M^{2n}$ are linearly independent. A $j$-independent GKM-graph is defined similarly. Properties of a GKM-action of a $T^k$-action on $M^{2n}$ in $j$-general position were studied in several papers (in a wider context of equivariantly formal smooth actions with isolated fixed points). The number $n-k$ is called the complexity of the $T^k$-action on $M^{2n}$. In \cite{ma-pa-06} it was shown that the equivariant cohomology ring of a torus manifold with vanishing odd cohomology groups (that is, a $2n$-dimensional GKM-manifold of complexity $0$ which is automatically $n$-independent) is isomorphic to the Stanley-Reisner ring of the face poset of the GKM-graph. Also see \cite{ma-ma-pa-07} for treatment of equivariant topology for torus manifolds from GKM-perspective. Furthermore, any GKM-manifold $M^{2n}$ of complexity $1$ with a GKM-action in general position (that is, in $(n-1)$-general position) has a description \cite{ay-ma-so-22} of the equivariant cohomology ring in terms of face rings similar to \cite{ma-pa-06}. The poset of faces $S_{M}$ in the orbit space $M/T$ was associated to a GKM-manifold $M$ with the $T$-action in \cite{ay-ma-so-22}. We remark that in the case of complexity $0$ this object was studied in earlier papers \cite{ma-pa-06} (for a locally standard action, so that $M/T$ is a manifold with corners) and in \cite{ay-ma-19} (for a general case, so that $M/T$ is a homological cell complex). It was shown in \cite{ay-ma-so-22} that for a GKM-manifold in $j$-general position the subposets $(S_{M})_{<s}$ and  $(S_{M})_{r}$ are $\min\lb \dim s-1,j+1\rb$- and $\min\lb r-1,j+1\rb$-acyclic, respectively, for any $r>0$ and $s\in S_{M}$. This implies by \cite{ma-pa-06, ay-ma-19} that the corresponding orbit space $M/T$ of $M$ is $(j+1)$-acyclic. This result was proved in \cite{ma-pa-06} for $j=n$ ($n$-independent case) and in \cite{ay-ma-19} for the case of arbitrary $j$. On the other hand, by dropping the condition of general position, an arbitrary $T^{n-1}$-action on a smooth manifold $M^{2n}$ (of complexity $1$) with isolated fixed points shows much more complicated behaviour of the orbit space homology \cite{ay-ch-19}. Therefore, GKM-actions in $j$-general position form a particularly nice class of GKM-actions exhibiting many useful properties (depending on the value of $j$).

A natural question (called an extension problem in \cite{ku-19}) is how to determine whether a given GKM-action extends to an effective GKM-action of a torus of greater dimension (so that the complexity of the new action is lower) on the same manifold, or not. This question leads to a search of a $k$-independent $(n,k)$-type GKM-action that is nonextendible for arbitrary $n\geq k\geq 2$. Homogeneous GKM-manifolds \cite{gu-ho-za-06} supply with many interesting examples of GKM-graphs, including those of positive complexity. Recall that if the Euler characteristic of a homogeneous complex manifold $G/H$ is nonzero then it is a GKM-manifold by \cite{gu-ho-za-06}, where $G$ is a compact connected semisimple Lie group, $H$ is an arbitrary parabolic subgroup in $G$ and $T$ is any maximal torus of $H\subset G$. Two phenomena occur for homogeneous GKM-manifolds related to the search problem mentioned above. Firstly, S.~Kuroki proved the following gap property for $G$ of type $A_n$ with the natural GKM $T$-action on $G/H$ as above (unpublished): if the respective GKM-graph is $j$-independent for some $j\geq 4$ then it is $n$-independent, where $n$ is the rank of $G$. This leads to a conjecture that any homogeneous GKM-manifold satisfies such a gap property. Secondly, the natural $T^{n-1}$-action on the Grassmanian $Gr_{k}(\Co^n)$ of $k$-planes in $\Co^n$ is a GKM-action admitting no nonisomorphic extensions for any $k$. This can be shown by modifying the proof for $k=2$ from \cite{ku-19} in a simple way. This result leads to another conjecture that most of the homogeneous GKM-manifolds are nonextendible with possibly a few exceptions determined by the condition that the respective automorphism group identity component $\Aut^0 (G/H)$ is greater than $G$. (The description of $\Aut^0 (G/H)$ is given in \cite{ak-95}, for instance.)

Therefore the extension problem seems to be a bit challenging. Instead, one can study a rather simplistic extension problem for arbitrary GKM-graphs by replacing a GKM-action of $T^{k}$ on $M^{2n}$ with an $(n,k)$-type $k$-independent GKM-graph $\Gamma$ (not necessarily coming from a GKM-action). An occurring new realization problem here is to determine whether a given GKM-graph $\Gamma$ is a GKM-graph of some GKM-manifold. We remark that some results on realization problem were obtained in dimensions $n=1,2$ in \cite{ca-ga-ka} for GKM-graphs (in terms of conditions given by the ABBV localization formula), and for GKM-orbifolds in dimension $n=2$ in \cite[Remark 4.13]{da-ku-so-18}.

In this paper, we give new examples of non-realizable and non-extendible actions of different complexity. The precise statement is given in the following main theorem of this paper (see Theorem \ref{thm:example_}). 

\begin{thmM}\label{thm:example}
For any $n\geq k\geq 3$ there exists an $(n,k)$-type GKM-graph $\Gamma$ satisfying the following properties:
\begin{enumerate}[label=(\roman*)]
\item The GKM-graph $\Gamma$ is $k$-independent;
\item The GKM-graph $\Gamma$ has no nontrivial extensions;
\item The GKM-graph $\Gamma$ is not realized by a GKM-action if $n=k=3$ or $n\geq k\geq 4$.
\end{enumerate}
\end{thmM}

The necessary tools for the proof of Theorem \ref{thm:example} are twofold. Firstly, in Section \ref{sec:Obstruction} we give a sufficient criterion (Corollary \ref{cor:specface}) of non-extendibility of any GKM-graph satisfying some condition formulated in terms of chords for a face in the GKM-graph. Secondly, in Section \ref{sec:lattices} for any $(j+1)$-independent GKM-manifold $M$ with the corresponding GKM-graph $\Gamma$, we compare (Proposition \ref{pr:relposets}) some subposets in the face poset $S_{\Gamma}$ of a GKM-graph $\Gamma$ with the subposets in the face poset $S_M$ of $M$. The obtained comparison implies (by \cite{ay-ma-so-22}) partial acyclicity for some subposets in $S_{\Gamma}$. In this case $(S_{\Gamma})_{\leq \Xi}^{op}$ turns out to be a simplicial poset for any face $\Xi$ of dimension not exceeding $j$ in $\Gamma$ (Proposition \ref{pr:condgenpos}). The respective Euler characteristic of the order complex for the poset $(S_{\Gamma})_{<\Xi}^{op}$ is computed in terms of the corresponding $f$-numbers (Proposition \ref{pr:euch}). 

In Section \ref{sec:periodic} we explicitly construct an example which is suitable for the proof of Theorem \ref{thm:example}. Here is a brief outline of the construction. Firstly we introduce an infinite $(d+1)$-regular graph embedded into $\Ro^d$. We endow it with an axial function in order to obtain a torus graph $\Gamma(d,0)$. We add to it edges and define the axial function agreeing to that of $\Gamma(d,0)$ causing the increase of the complexity for a GKM-graph. The resulting $(d+1+r,d+1)$-type GKM graph $\Gamma(d,r)$ is $(d+1)$-independent. We remark that the definition of the axial function on $\Gamma(d,r)$ is obtained by applying the decision method of Tarski \cite{ta-51} to some Vandermonde matrices (Lemma \ref{lm:findvals}) and therefore it is implicit. The infinite GKM graph $\Gamma(d,r)$ is periodic (invariant) with respect to the subgroup $2^{r+1}\cdot\Zo^d\subset\Zo^d$ in the group of parallel translations in $\Ro^{d}$. The quotient $\Gamma_{d+1}^{d+1+r}:=\Gamma^{2^{r+1}}(d,r)$ of the GKM-graph $\Gamma(d,r)$ by the group $2^{r+1}\cdot\Zo^d$ is shown to be a well-defined GKM-graph without multiple edges and loops.

In the concluding Section \ref{sec:proof} of this paper we prove that $\Gamma_{d+1}^{d+1+r}$ satisfies all conditions of Theorem \ref{thm:example} (where we put $n=d+1+r$, $k=n+1$). We apply the results on acyclicity of face subposets in GKM-manifolds (in the case of complexity $0$) from \cite{ma-pa-06} in order to prove non-realizability of the constructed GKM-graph $\Gamma_{d+1}^{d+1+r}$ by studying the respective Euler characteristic (by the comparison results mentioned above). This argument relies on the explicit computation of face numbers in the torus graph $\Gamma^{a}(d,0)$ (Lemma \ref{lm:eucomp}). The nonextendibility is proved by using the method of chords mentioned above.

\section{An obstruction to a GKM-graph extension}\label{sec:Obstruction}

In this section we recall some definitions from GKM-theory (we follow the notation from \cite{gu-za-01} and \cite{ku-09}). We introduce an obstruction to have extensions for a given GKM-graph in terms of chords.

\begin{defin}\cite{gu-za-01}\label{def:gkm}
A \textit{GKM-graph} $\Gamma$ is a triple $((V,E),\nabla,\alpha)$ consisting of:
\begin{itemize}
\item a graph $(V,E)$ with the set of vertices $V\neq \varnothing$ and with the set of edges $E\subseteq V\times V\setminus \Delta(V)$, where $\Delta(V):=\lb (v,v)|\ v\in V\rb\subseteq V\times V$ and it is required that for an edge $e=(u,v)\in E$ one has $\overline{e}:=(v,u)\in E$;\\
\item a collection of bijections 
\[
\nabla=\lb \nabla_{e}\colon \str_{\Gamma} i(e) \to \str_{\Gamma} t(e)\rb,
\]
such that the identity
\[
\nabla_{e}(e)=\overline{e},
\]
holds for any $e\in E$, where $\str_{\Gamma} i(e):=\lb e\in E|\ i(e)=v\rb$
is the \textit{star} of $\Gamma$ at $i(e)$;

\item a function $\alpha\colon E\to \Zo^k$ satisfying
the \textit{rank condition}
\[
\alpha\la {\Gamma}\ra:=\Zo\la \alpha(e)|\ e\in \str_{\Gamma}(v)\ra = \Zo^k,
\]
the \textit{opposite sign condition}
\[
\alpha(\overline{e})=-\alpha(e),
\]
and the \textit{congruence condition}
\[
\alpha(\nabla_{e} (e'))=\alpha(e')+c_{e}(e')\alpha(e),
\]
for some integer $c_{e}(e')\in\Zo$ and for any $e,e'\in E$ with a common source, and $v\in V$.
\end{itemize}
The collection $\nabla$ is called a \textit{connection} of $\Gamma$ and the function $\alpha$ is called an \textit{axial function} of $\Gamma$.
\end{defin}

In this section we fix a GKM-graph $\Gamma$ with the corresponding connected $n$-valent graph $(V,E)=(V_{\Gamma},E_{\Gamma})$, axial function $\alpha\colon E\to\Zo^k$ on $\Gamma$ and a connection $\nabla$ on it. 

\begin{defin}\cite{gu-za-01}
A connected $r$-regular subgraph $\Xi$ of $\Gamma$ is called an \textit{$r$-face} (or a \textit{face}) of $\Gamma$, if $\nabla_{e}(e')\in E_{\Xi}$ holds for any $e,e'\in E_{\Xi}$ such that $i(e)=i(e')$ (in \cite{gu-za-01} it is called a \textit{totally geodesic subgraph}). Any edge $e\in \str_{\Gamma} v\setminus \str_{\Xi} v$ is called a \textit{transversal edge} to a face $\Xi$ in $\Gamma$, where $v\in V_{\Xi}$. Let
\[
\alpha\la \Xi\ra=\alpha_{v}\la \Xi\ra:=\Zo\la \alpha(e)|\ e\in \str_{\Xi}(v)\ra\subseteq \Zo^k,
\]
be a \textit{span} of a face $\Xi$ in $\Gamma$, where 
$\Zo\la \alpha(e)|\ e\in \str_{\Xi}(v)\ra$ denotes the $\Zo$-linear span of vectors $\alpha(e)$, where $e$ runs over $\str_{\Xi}(v)$. The GKM-graph $\Gamma$ is called a GKM-graph of \textit{$(n,k)$-type}, if $\Gamma$ is $n$-valent and the \textit{rank} $\rk \alpha:=\rk\alpha\la\Gamma\ra$ of $\alpha$ is equal to $k$.
\end{defin}

\begin{rem}
A face $\Xi$ of a GKM-graph $\Gamma$ becomes a well-defined GKM-graph by taking restrictions of the connection and of the axial function from $\Gamma$ to $\Xi$. The span of the face $\Xi$ is well defined because of the identity $\alpha_{p}\la \Xi\ra= \alpha_{q}\la \Xi\ra$ for every $p,q\in V_{\Xi}$ which immediately follows from the congruence condition.
\end{rem}

\begin{defin}\label{defin:chord}
A GKM-graph $\Gamma$ is called {\it $j$-complete} if for any $v\in V$, any integer $i\leq j$ and any distinct edges $e_1,\dots,e_i\in \str_{\Gamma}(v)$ there exists an $i$-face $\Xi$ of $\Gamma$ such that $\str_{\Xi}(v)=\{e_1,\dots,e_i\}$ holds. An $n$-regular $n$-complete GKM-graph is called a \emph{complete} GKM-graph. A GKM-graph $\Gamma$ is called \textit{$j$-independent} if for any $v\in V$ and any distinct edges $e_1,\dots,e_j\in\str_{\Gamma} v$ the values $\alpha(e_1),\dots,\alpha(e_j)$ of the axial function $\alpha$ on $\Gamma$ are linearly independent in $\Zo^k$.
\end{defin}

\begin{defin}
Let $\Xi$ be a face of $\Gamma$. We call a transversal edge $e\in E_{\Gamma}$ to $\Xi$ \textit{a chord} of the face $\Xi$, if $i(e),t(e)\in V_{\Xi}$. If the face $\Xi$ in $\Gamma$ admits no chords then we call $\Xi$ \textit{a chordless face} of $\Gamma$.
\end{defin}

\begin{ex}\label{ex:t2fl3}
The standard $T^2$-action on the flag manifold $\mathcal{F}l(3)$ is a GKM-action with the GKM-graph $\Gamma$ of $(3,2)$-type and with $K_{3,3}$ as the underlying graph \cite[p.40]{gu-ho-za-06}. The GKM-graph $\Gamma$ satisfies the opposite sign condition. It has five $2$-faces (three $4$-cycles and two $6$-cycles). For any such $6$-cycle $2$-face $\Xi$ the remaining $3$ transversal edges in $\Gamma$ are chords of $\Xi$. 
\end{ex}

\begin{prop}\label{pr:omitcond}
Let $\Xi$ be a $j$-face of a $(j+1)$-complete GKM-graph $\Gamma$. Then for any chord $e\in E_{\Gamma}$ of $\Xi$ and any edge path $\gamma=(e_1,\dots,e_r)$ in $\Xi$ such that $i(\gamma)=i(e)$, $t(\gamma)=t(e)$, one has $\Pi_{\gamma} e=\overline{e}$, where by definition (see \cite{ta-04})
\[
\Pi_{\gamma}:\ \str_{\Gamma}(i(\gamma))\to \str_{\Gamma}(t(\gamma)),\quad \Pi_{\gamma}(e):=\nabla_{e_r}\circ\dots\circ\nabla_{e_1} (e).
\]
\end{prop}
\begin{proof}
By the $(j+1)$-completeness condition, there exists a $(j+1)$-face $\Phi$ of $\Gamma$ such that $\Xi$ and $e$ belong to $\Phi$. Notice that $\Pi_{e}(e)=\overline{e}\in \Phi$ holds by the definition. The definition of invariance also implies that $\Pi_{e}(e)\in \str_{\Phi}(t(\gamma))\setminus \str_{\Xi}(t(\gamma))$ holds. However, $\str_{\Phi}(t(\gamma))\setminus \str_{\Xi}(t(\gamma))$ has cardinality one, since $\Xi$ and $\Phi$ are $j$- and $(j+1)$-faces, respectively. Hence,
\[
\str_{\Phi}(t(\gamma))\setminus \str_{\Xi}(t(\gamma))=\lb \overline{e}\rb,
\]
holds. Observe that $\Pi_{\gamma}(e)\in \str_{\Phi}(t(\gamma))\setminus \str_{\Xi}(t(\gamma))$ holds by the definition of invariance, because $\gamma$ belongs to $\Xi$ by the condition. Therefore, we conclude that $\Pi_{\gamma}(e)=\overline{e}$ holds. This proves the claim of the proposition.
\end{proof}

\begin{prop}\label{pr:labelinspan}
For a chord $e\in E_{\Gamma}$ of a face $\Xi$ in $\Gamma$ suppose that there exists an edge path $\gamma$ in $\Xi$ such that $i(\gamma)=i(e)$, $t(\gamma)=t(e)$ and $\Pi_{\gamma} e=\overline{e}$ hold. Then one has $2\alpha(e)\in\alpha\la \Xi\ra$.
\end{prop}
\begin{proof}
Let $\gamma=(e_1,\dots,e_r)$. Notice that one has $\Pi_{\gamma_{i-1}} e_i\in E_{\Xi}$, because $\Xi$ is a face in $\Gamma$, where $\gamma_{i}:=(e_1,\dots,e_i)$ and $\gamma_0:=i(\gamma)$, $i=1,\dots,r$. Hence, $\alpha(\Pi_{\gamma_{i-1}} e_i)\in \alpha\la \Xi\ra$ holds for any $i=1,\dots,r$. One deduces the identity
\begin{equation}\label{eq:compforchord}
\alpha(\Pi_{\gamma}e)=\alpha(e)+\sum_{i=1}^{r} c_{e_i}(\Pi_{\gamma_{i-1}} e) \cdot \alpha(\Pi_{\gamma_{i-1}} e_i),
\end{equation}
from the congruence condition. We conclude that
\begin{equation}\label{eq:diffspan}
\alpha(\Pi_{\gamma}e)-\alpha(e)\in \alpha\la \Xi\ra,
\end{equation}
holds. Notice that the identities $\alpha(\overline{e})=-\alpha(e)$, $\alpha(\overline{e})=\alpha(\Pi_{\gamma}e)$ hold. Together with the inclusion \eqref{eq:diffspan} this implies the claim of the proposition.
\end{proof}

\begin{prop}\label{pr:invsubg}
Suppose that the underlying graph of $\Gamma$ has finitely many vertices. Then, if $\Gamma$ is $(j+1)$-independent, then $\Gamma$ is $j$-complete, where $j\in\Zo$.
\end{prop}
\begin{proof}
Fix a nonzero integer $s\leq j$. Let $E':=\lb e_1,\dots, e_{s}\rb$ be an $s$-element set of some mutually different edges in $\Gamma$ with a common origin $v$. In order to prove the claim it is enough to construct an $s$-face $\Xi$ in $\Gamma$ such that the inclusion
\begin{equation}\label{eq:facecond}
E'\subseteq E_{\Xi}, 
\end{equation}
holds. We give the inductive definition as follows:
\[
P_{i+1}:=P_{i}\cup \lb \Pi_{e} e'|\ e,e'\in P_{i}\rb,\ P_{0}:=E',\ i\geq 0.
\]
By the definition, the filtration $P_0\subseteq P_{1}\subseteq\cdots$ is bounded by the finite set $E_{\Gamma}$ from above. Hence, there exists $N\in \No$ such that $P_{i}=P_{N}$ holds for any $i\geq N$. 
Define the subgraph $\Xi$ in $\Gamma$ by the formulas
\[
V_{\Xi}:=\lb i(e)|\ e\in P_{N}\rb,\ E_{\Xi}:=P_{N}.
\]
The set $P_{N}$ is closed under reversion of an edge operation, because $\Pi_{e}(e)=\overline{e}\in P_{i+1}$ holds for any $e\in P_{i}$. By the condition, for any $e,e'\in P_{N}$ there exists an edge path $\gamma\subseteq\Xi$ such that $i(\gamma)=i(e)$ and $t(\gamma)=i(e')$ holds. Hence, $\Xi$ is a connected subgraph in $\Gamma$. It follows from the definition that $\Xi$ is a face of $\Gamma$. It remains to show that $\Xi$ is an $s$-face. Assume the contrary. Then there exists $e\in \str_{\Xi}(v)\setminus E'$. It follows from the definition that there exist $i=1,\dots,s$ and $\gamma\subseteq \Xi$ such that $i(\gamma)=t(\gamma)=v$ and 
\[
\Pi_{\gamma}(e')=e,
\]
holds. It follows from the formula \eqref{eq:compforchord} that 
\[
\alpha(\Pi_{\gamma}(e'))\in\Zo\la \alpha(e_j)|\ j=1,\dots,s\ra.
\]
Hence, the collection of $s+1$ vectors $\alpha(e),\alpha(e_j)$, $j=1,\dots,s$, is linearly dependent. However, this contradicts the condition of $(j+1)$-independency of $\Gamma$, because $s\leq q$. We conclude that $\Xi$ is an $s$-face, which proves the claim of the proposition.
\end{proof}

\begin{cor}\label{cor:nonextend}
Suppose that the underlying graph of $\Gamma$ has finitely many vertices. Then, if $\Gamma$ is a $(j+2)$-independent ($n$-independent, respectively) GKM-graph for some $j\in\Zo$, then any $r$-face (face, respectively) of $\Gamma$ is chordless, where $r=1,\dots,j$.
\end{cor}
\begin{proof}
Assume the contrary. Then there exist an $r$-face $\Xi$ of $\Gamma$ and its chord $e$, where $r\leq j$. One has $2\alpha(e)\notin \alpha\la \Xi\ra$, because $\Gamma$ is $(j+2)$-independent and $\Xi$ is $r$-regular, where $r\leq j$. By Proposition \ref{pr:invsubg}, $\Gamma$ is $(j+1)$-complete. Then one can apply Propositions \ref{pr:omitcond} and \ref{pr:labelinspan} in order to obtain $2\alpha(e)\in \alpha\la \Xi\ra$. This contradiction proves the first claim of the corollary. The proof of the second claim is similar to the proof of the first claim.
\end{proof}

\begin{defin}
Let $\Gamma'$, $\Gamma$ be two GKM-graphs with the same underlying graph $(V,E)$, the same connection $\nabla$, with the axial functions $\alpha'$, $\alpha$ taking values in $\Zo^{k'}$ and $\Zo^k$, respectively. The GKM-graph $\Gamma'$ is called an \textit{extension} of $\Gamma$ (see \cite{ku-19}), if there exists an epimorhism $p\colon \Zo^{k'}\to \Zo^k$ such that $p(\alpha'(e))=\alpha(e)$ holds for any $e\in E$. We say that an $(n,k)$-type GKM-graph $\Gamma$ \textit{has no nontrivial extensions} if for any $s>0$ it does not admit an extension to an $(n,k+s)$-type GKM-graph. (This terminology was proposed by S.~Kuroki.)
\end{defin}

The next corollary is a principal tool for the proof of Theorem \ref{thm:example}.

\begin{cor}\label{cor:specface}
Suppose that the underlying graph of $\Gamma$ has finitely many vertices. Then, if $\Gamma$ is a $(k+1)$-complete GKM-graph and there exists a $k$-face $\Xi$ of $\Gamma$ such that any transversal edge $e\in E_\Gamma$ to $\Xi$ is a chord for $\Xi$, then $\Gamma$ has no nontrivial extensions.
\end{cor}
\begin{proof}
Suppose that there exists an extension of $\alpha$ to an axial function $\widetilde{\alpha}$ of rank $k+s$ for some $s>0$. Choose a vertex $v\in V_{\Xi}$. Then it follows from the condition by Propositions \ref{pr:omitcond} and \ref{pr:labelinspan} that $2\widetilde{\alpha}(e)\in \widetilde{\alpha}\la \Xi\ra$ holds for any edge $e\in \str_{\Gamma}(v)$. Hence, $k+s=\rk \widetilde{\alpha}= \rk \widetilde{\alpha}\la \Xi\ra$. However, by definition $\rk \widetilde{\alpha}\la \Xi\ra\leq k$. This contradiction proves the claim.
\end{proof}

\section{Face posets of a GKM-graph and of a GKM-manifold}\label{sec:lattices}

In this section we continue to recall some basic notions of GKM-theory and of the related \cite{ma-ma-pa-07, ay-ma-so-22} posets $S_{M}$, $S_\Gamma$ of faces arising from the orbit space $M/T$ and from the GKM-graph $\Gamma$ of a given GKM-manifold $M$ with the $T$-action, respectively. We compare some specific simplicial subposets in $S_{M}$, $S_\Gamma$ under assumption of $j$-general position for $M$. After that we recall the P.Hall formula for the Euler characteristic of an order complex for a finite simplicial poset which is used later in the text.

\begin{defin}\cite{ma-ma-pa-07,ay-ma-so-22}
For a GKM-graph $\Gamma$ the collection $S_{\Gamma}$ of all faces in $\Gamma$ is called a \textit{face poset of the GKM-graph} $\Gamma$ with the partial order given by inclusion of faces in $\Gamma$. 
\end{defin}

Due to \cite[Lemma 2.1]{ma-pa-06} one can give the following definition of a GKM-manifold that is equivalent to the standard one (e.g. see \cite{gu-za-01}).

\begin{defin}\cite{gkm-98,gu-za-01,ma-pa-06}
A smooth manifold $M^{2n}$ with an effective action of $T^k=(S^1)^k$ is called a \textit{GKM-manifold} if the following conditions hold:
\begin{itemize}
\item the set of $T^k$-fixed points $M^T$ in $M$ is finite and nonempty;\\
\item the tangent weights of the $T^k$-action at any $x\in M^T$ are pairwise linearly independent;\\
\item all odd cohomology groups of $M$ vanish, i.e. one has $H^{odd}(M;\ \Zo)=0$.
\end{itemize}
\end{defin}

\begin{rem}
To any complex GKM-manifold one associates a GKM-graph, e.g. see \cite{ku-09}. We notice that for an arbitrary GKM-manifold the opposite sign condition is in general satisfied only up to a sign. We also remark that it is possible to have loops and multiple edges for a GKM-action. In this paper we restrain from considering such torus actions and we use a restricted definition of a GKM-graph (where it is a simple graph). Let $T'$ and $T$ be two GKM-actions of tori on the same manifold $M$. The action of $T'$ is called an \textit{extension} of the action $T$ on $M$ if there is a group monomorphism $\pi\colon T\to T'$ that is equivariant with respect to these torus actions. In other words, the $T$-action is the restriction of the $T'$-action. The epimorphism 
\[
p\colon \Zo^{k'}\cong\Hom(T',S^1)\to \Hom(T,S^1)\cong\Zo^{k},
\]
corresponding to $\pi$ induces the extension of the GKM-graphs $\Gamma'$, $\Gamma$ corresponding to the $T'$- and the $T$-action, respectively.
\end{rem}

\begin{ex}
The natural $T^2$-action on $\mathcal{F}l_{3}$ has no notrivial extensions by proving that for the corresponding GKM-graph by Corollary \ref{cor:specface} (see Example \ref{ex:t2fl3}). This fact may also be easily obtained by the results of \cite{ku-19}, or by studying the automorphism group of the homogeneous space $\mathcal{F}l_{3}$ (in a different category of complex-analytic torus actions).
\end{ex}

Consider a GKM-action of $T=T^k$ on $M=M^{2n}$.

\begin{defin}\label{defn:facetop}\cite{ay-ma-so-22}
For a smooth $T$-action on $M$, consider the canonical projection $p\colon M\to Q:=M/T$ to the respective orbit space, and let
\begin{equation}\label{eq:orbf}
Q_0\subset Q_1\subset\cdots \subset Q_k=Q
\end{equation}
\[
Q_i:=p(M_i),\ M_i=\{x\in M\colon \dim Tx\leqslant i\},
\]
be the filtration on the orbit space $Q$, where $Tx$ denotes the $T$-orbit of $x$ in $M$. The closure of a connected component of $Q_i\setminus Q_{i-1}$ is called an \emph{$i$-face} (or a \textit{face}) $F$ of $Q$ if it contains at least one fixed point. 
\end{defin}

\begin{defin}\cite{ay-ma-so-22}
The \textit{poset of faces for the GKM-manifold} $M$ is the poset $S_{M}$ of faces of nonnegative dimension ordered by inclusion in the orbit space $Q$ of the $T$-manifold $M$. 
\end{defin}

Recall that a topological space $X$ is called \textit{$j$-acyclic} if $\tilde{H}^{i}(X)=0$ holds for any $i\leq j$, and \textit{acyclic}, if $\tilde{H}^{*}(X)=0$ holds. Let $S_{M}$ be the poset of faces for a GKM-manifold $M$. Let $S_{\Gamma}$ be the poset of faces of nonnegative dimension (ordered by inclusion) of a GKM-graph $\Gamma$. We need the following particular case of a theorem from \cite{ay-ma-so-22}.

\begin{thm}\cite[Theorem 1]{ay-ma-so-22}\label{thm:mp}
For any GKM-manifold $M$ of complexity $0$ in $n$-general position (that is, $k=n$ holds) the $(n-1)$-dimensional poset $\overline{S_{M}^{op}}$ is $(n-2)$-acyclic, where $n\geq 2$.
\end{thm}

\begin{lem}\cite[p.5, Lemma 2.9]{ay-ma-so-22}\label{lm:ay2}
The full preimage $M_F=p^{-1}(F)$ of any face $F\subseteq Q$ is a smooth submanifold in $M$ called a face submanifold in $M$.
\end{lem}

The claim of the following proposition is reminiscent to \cite[Lemma 3.8]{ay-ma-so-22} (although not quite the same).

\begin{prop}\label{pr:relposets}
For a GKM-manifold $M$ in $(j+1)$-general position for some $j\geq 1$ the following claims hold. 

$(i)$ For any $q\leq j$, any $q$-face $\Xi$ in $\Gamma$ is an equivariant $1$-skeleton of a face submanifold in $M$ and the GKM-graph $\Xi$ is a torus graph. 

$(ii)$ The span $\alpha\la \Xi\ra$ of $\Xi$ splits off as a direct factor in $\Zo^k$.

$(iii)$ The posets $(S_{M})_{\leq s(\Xi)}$ and $(S_{\Gamma})_{\leq \Xi}$ are isomorphic for any face $\Xi$ of $\Gamma$ such that $\dim \Xi\leq j$, where $s(\Xi)$ is the face in $M$ corresponding to $\Xi$ by $(i)$.
\end{prop}
\begin{proof}
Choose $v\in V_{\Xi}$ and let $\str_{\Xi}(v)=\lb e_1,\dots,e_q\rb$. Let $G$ be a closed subgroup in $T$ corresponding to the sublattice $L:=\Zo\la \alpha_{i}|\ i=1,\dots,q\ra$ in $\Zo^k$, where $\alpha_{i}:=\alpha(e_i)$. Let $G_0$ be the identity component (in particular, a torus) of $G$. Notice that the sublattice $L_0\subset \Zo^k$ corresponding to the subtorus $G_{0}\subset T$ splits off as a direct factor in $\Zo^k$ and that there is a lattice embedding $L\subseteq L_0$ of a finite index. The connected component $Y$ of $M^{G_0}$ such that $v\in Y$ is a smooth manifold with effective $T'$-action by Lemma \ref{lm:ay2}, where $T':=T/G_{0}$. Notice that $Y^{T'}\subseteq M^{T}$. The set of weights $\beta_{1},\dots,\beta_{r}$ of the $T'$-action on $Y$ at $v$ embed to the set of weights of the $T$-action on $M$ at $v$. One has $\Zo\la \beta_{j}|\ j=1,\dots,r\ra=L_0$. Hence, $\beta_j,\alpha_1,\dots,\alpha_q$ are linearly dependent for any $j=1,\dots,r$. Then by linear independence condition we conclude that $\lb \beta_{1},\dots,\beta_r\rb=\lb \alpha_1,\dots,\alpha_q\rb$ holds. Therefore, $\dim Y=2q$, $Y$ is a GKM-manifold, its equivariant $1$-skeleton $Y_1$ is a GKM-graph $\Xi$ of type $(q,q)$ and $\Xi$ is a face of the GKM-graph $\Gamma$. Notice that this implies $L=L_0$. Hence, the claims $(i)$, $(ii)$ are proved. By the definition, one has $(S_{M})_{\leq s}\subseteq (S_{\Gamma})_{\leq s}$. The inverse inclusion holds by $(i)$. This proves $(iii)$. The proof is complete.
\end{proof}

Let $P$ be a finite poset \cite{st-86}. Recall the following definitions.

\begin{defin}\cite{st-86}
The \textit{order complex} of a finite poset $P$ is the simplicial complex 
\[
\Delta (P):=\lb \sigma=\lb I_1, I_2,\dots, I_{k+1}\rb\in 2^{P}|\ 
I_1< I_2<\cdots< I_{k+1},\ k\geq 0\rb,
\]
on the vertex set $P$ consisting of chains of increasing elements in $P$. By definition the $q$-faces of a simplex $\sigma$ are $I_{i_1}< I_{i_2}<\cdots< I_{i_{q+1}}$, where $1\leq i_1<\cdots<i_{q+1}\leq k$ are arbitrary numbers (i.e. the chains obtained by dropping the elements of $\sigma$), $0\leq q\leq k$. 
\end{defin}

\begin{defin}\cite{st-86}
The poset $P$ with the least element $\hat{0}$ is called a \textit{simplicial poset} if the subposet $[\hat{0},x]$ of $P$ is a Boolean lattice for any $x\in P$. For any element $x$ of a simplicial poset $P$ a \textit{length} $l(x)$ of $x$ is the length of a maximal chain in $[\hat{0},x]$. Here $l(\hat{0}):=0$. Define the \emph{dimension} of a simplicial poset $P$ to be the number $\dim P:=\dim \Delta(P)=\max_{x\in P} l(x)$. For a simplicial poset $P$ let 
\[
f_{i}(P):=|\lb x\in P|\ l(x)=i+1\rb|,
\]
be the number of elements in $P$ of length $i+1$, where $i\geq 0$. In particular, $f_{-1}(P)=1$.
\end{defin}

\begin{rem}
The poset $S_{\Gamma}^{op}$ has the least element $\Gamma$ by the definition and is therefore acyclic. For a torus action with a dense open orbit the poset $S_{M}^{op}$ has the least element and is contractible, too. However, for an arbitrary $T$-action on $M$ the poset $S_{M}^{op}$ is neither acyclic nor simplicial, in general. This is due to the fact that the orbit space $Q$ is a homological cell complex and the group isomorphism $H^*(\Delta(S_{M}))\cong H^*(Q)$ holds, e.g. see \cite[Proposition 5.14]{ma-pa-06}. For instance, it can be checked for the $T^2$-action on $\mathcal{F}l(3)$ (Example \ref{ex:t2fl3}) that the corresponding orbit space is homeomorphic to a sphere $Q\cong S^4$, e.g. see \cite{ay-ma-19}.
\end{rem}

In the following we need the following well-known Philip Hall's theorem.

\begin{prop}\cite[Proposition 6]{ro-64}\label{pr:euch}
Let $S$ be a simplicial poset of dimension $d$. Then the Euler characteristic $\tilde{\chi}(\Delta (\overline{S}))$ of the order complex for $\overline{S}:=S\setminus \hat{0}$ in the reduced simplicial homology $\tilde{H}_*(\Delta (\overline{S}))$ is given by the formula:
\[
\tilde{\chi}(\Delta (\overline{S}))=\sum_{i=-1}^{d-1} (-1)^i f_{i}(S).
\]
\end{prop}

The computation of Euler characteristic for certain face subposets in $S_{\Gamma}$ for a $j$-complete GKM-graph $\Gamma$ is possible (by using Proposition \ref{pr:euch}) due to the following proposition.

\begin{prop}\label{pr:condgenpos}
Let $\Gamma$ be a $j$-complete GKM-graph for some $j\geq 1$. Then for any $j$-face $\Xi$ of $\Gamma$ the poset $(S_{\Gamma})_{\leq \Xi}^{op}$ is a simplicial poset of dimension $\dim\Xi$. In particular, for any $\Omega\in (S_{\Gamma})_{\leq \Xi}^{op}$ one has $l(\Omega)=j-\dim \Omega$ in $(S_{\Gamma})_{\leq \Xi}^{op}$, and $f_{i}((S_{\Gamma})_{\leq \Xi}^{op})$ is equal to the number of $(j-i-1)$-dimensional faces in $(S_{\Gamma})_{\leq \Xi}$.
\end{prop}
\begin{proof}
In order to prove the claim of the proposition it is enough to show that for any face $\Omega$ in $\Xi$ the poset $[\Omega,\Xi]=\lb \Phi\in S_{\Gamma}|\ \Omega\subseteq \Phi\subseteq \Xi\rb$ is isomorphic to the poset of faces in $\Delta^{j-\dim \Omega}$. Let $v\in V_{\Omega}$. Since the connection of $\Gamma$ is $j$-independent, any face $\Phi\in [\Omega,\Xi]$ is uniquely determined by the collection $C(\Phi)$ of $\dim \Phi-\dim \Omega$ mutually different elements from $\str_{\Xi} v\setminus\str_{\Omega} v$, and vice versa. Moreover, for any $\Phi_1,\Phi_2\in [\Omega,\Xi]$ one has $\Phi_1\subseteq \Phi_2$ iff $C(\Phi_1)\subseteq C(\Phi_2)$. This implies the necessary claim. The proof is complete.
\end{proof}

\section{A periodic GKM-graph and its quotient}\label{sec:periodic}

In this section we give a detailed construction of the GKM-graph suitable for the proof of Theorem \ref{thm:example} and study some of its properties. 

\begin{con}[Graph $\Gamma'$]\label{constr:gamma'}
Let $I^{d}_{R}(x)$ be the edge graph of the cube
\[
\Io^{d}_{R}(x):=\lb y=(y_1,\dots,y_d)\in \Ro^d|\ |x_i-y_i|\leq R\rb,
\]
with center at $x=(x_1,\dots,x_d)\in\Ro^d$ and with edges of length $2R$. For any $d\geq 1$ we define the graph $\Gamma'=\Gamma'(d)$ embedded into $\Ro^d$ as the union of the following graphs:
\begin{enumerate}[label=(\roman*)]
\item Graph $I^{d}_{1/6}(x)$, where $x$ runs over all points $\Zo^d\subset\Ro^d$ with integral coordinates;
\item Graph $I^{d}_{1/6}(x+\frac{1}{2}\cdot(1, \dots, 1))$, where $x$ runs over $\Zo^d$;
\item A \textit{diagonal}, that is, an edge of the form
\[
D(x,u):=(x+\frac{1}{6} \cdot\sum_{i=1}^d (-1)^{u_i} e_i,\ x+\frac{1}{3}\cdot \sum_{i=1}^d (-1)^{u_i} e_i),
\]
(and its inverse), where $x$ runs over $\Zo^d$, $u=(u_1,\dots,u_d)$ runs over $\lb \pm 1\rb^d$, and $e_1,\dots,e_d$ is the standard basis of $\Ro^d$.
\end{enumerate}

Notice that the graph $\Gamma'$ is a $(d+1)$-regular connected graph with infinite set of vertices. We call a \textit{cubical subgraph} any subgraph in $\Gamma'$ of the form $I_{1/6}^{d}(x)$, where $x\in L_d:=\Zo^d\sqcup(\Zo^d+\frac{1}{2}(1,\dots,1))\subset \Ro^d$. For any vertex $v\in V_{\Gamma'}$ denote by $\Cube(v)$ a unique cubical subgraph in $\Gamma'$ such that $v\in V_{\Cube(v)}$ holds.
\end{con}

\begin{con}[Functions $\varepsilon^{i}_{j}$]\label{constr:sf}
For any $d\geq 1$ define the functions $\varepsilon_{i}^{j}\colon V_{\Gamma'}\to \lb\ \pm 1\rb$, where $i=1\dots,d+1$ and $j\in\No$. For any $x=(x_1,\dots,x_d)\in\Zo^d$ and any vertex $y$ of $I^{d}_{1/6}(x)$ let
\begin{equation}\label{eq:epsdef}
\varepsilon_{i}^{j}(y):=(-1)^{\floor{\frac{x_i}{2^{j-1}}}};\ i=1,\dots,d;\ j\in\No.
\end{equation}
By definition, the function $\varepsilon_{i}^{j}$ is then uniquely defined by taking the same values on the vertices of any diagonal of $\Gamma'$, where $i=1,\dots,d$ and $j\in\No$. Define
\begin{equation}\label{eq:epsdeflast}
\varepsilon_{d+1}^{j}(y):=(-1)^{\floor{\frac{y_1}{2^{j-1}}}},
\end{equation}
for any vertex $y=(y_1,\dots,y_d)\in\Zo^d$ of $\Gamma$.
\end{con}

\begin{con}[Graph $\Gamma$]\label{constr:gamma}
For any $d\geq 1$ and $r\geq 0$ let $\Gamma=\Gamma(d,r)$ be the graph obtained from $\Gamma'(d)$ by adding the edges $E_{j}(v):=(v,v')$, where $v'=v+2^{j-1}\cdot(1,\dots,1)$, where $j$ runs over $1,\dots,r$ and $v$ runs over the subset of elements $u=(u_1,\dots,u_d)\in V_{\Gamma'}$ such that $(-1)^{\floor{\frac{\sum_i u_i}{2^{j-1}}}}=1$ holds. 
\end{con}

Notice that the graph $\Gamma$ is a $(d+1+r)$-valent connected graph with infinite set of vertices, and that $V_{\Gamma}=V_{\Gamma'}$ holds (see Figure \ref{fig:cross}, \ref{fig:edges}).

\begin{con}[Axial function $\alpha$ on $\Gamma$]\label{constr:gammaax}
Fix a collection of integers $t_1,\dots,t_r\in\Zo$. Let $\alpha\colon E_{\Gamma}\to \Zo^{d+1}$, $\alpha=\alpha(d,r,t_1,\dots,t_r)$ be the function taking value $\alpha(D(x,u))=e_{d+1}$ for any $x\in\Zo^d$ and any $u\in\lb \pm 1\rb^d$. By definition, for any $e\in E_{I^{d}_{1/6}(x)}$ let $\alpha(e)$ be the inner (outer, respectively) normal of the unit length for the corresponding to $e$ facet of the cube $\Io^{d}_{1/6}(x)$ in $\Ro^d$ if $x\in \Zo^d$ (if $x\in L_d \setminus\Zo^d$, respectively). For any diagonal $e=D(x,u)$, where $x\in\Zo^d$ and $u\in\lb \pm 1\rb^d$, let $\alpha(e)=e_{d+1}$, $\alpha(\overline{e})=-e_{d+1}$. In particular, one has $\alpha(e)\in \lb \pm e_1,\dots,\pm e_{d+1}\rb$ for any $e\in E_{\Gamma'}$ (see Figure \ref{fig:square}). For any $v\in V_{\Gamma}$ let
\begin{equation}\label{eq:axvals}
\alpha(E_{j}(v))=\sum_{i=1}^{d+1} \varepsilon_{i}^{j}(v) t_{i}^{j-1} w_i(v),
\end{equation}
for any $j\in\No$, where $\lb w_1(v),\dots,w_{d+1}(v)\rb$ are the values of $\alpha$ on $\str_{\Gamma'}(v)$ denoted in such a way that $w_i(v)=\pm e_i$ holds for $i=1,\dots,d+1$. 
\end{con}

Denote by $A_{x}$ the automorphism of the graph $\Gamma(d,r)$ induced by the linear operator $y\mapsto x+y$ in $\Ro^d$, where $x\in L_d$; $y\in \Ro^d$. Notice that $A_x$ is well defined for any $x\in L_d$ and that the identities
\begin{equation}\label{eq:opp}
\varepsilon_{i}^{j}(A_{2^{j-1}\cdot v}(x))=-\varepsilon_{i}^{j}(x),
\end{equation}
\begin{equation}\label{eq:transax}
\alpha(A_{1/2\cdot u}(e))=-\alpha(e),
\end{equation}
hold for any $x\in V_{\Gamma}$, $i=1,\dots,d+1$; $j\in\No$; $e\in E_{\Gamma}$; $u\in \lb \pm 1\rb^d$ and any $v\in\Zo^d$.

Our next task is to define the connections $\nabla'$, $\nabla$ on $\Gamma'$ and $\Gamma$ compatible with $\alpha'$ and $\alpha$ by describing the corresponding facets, respectively. We do this by listing all facets in the corresponding graphs in the next two definitions. One can easily check that the facets given below are compatible with $\alpha'$ and $\alpha$.

\begin{defin}[Facets of $\Gamma'$]
For any $v\in V_{\Gamma'}$ let $F_0(v):=\Cube(v)$ be the subgraph in $\Gamma'$. Denote by $\Cube_i (v)$ a unique subgraph in $\Cube(v)$ corresponding to the facet of the respective cube $\Io_{1/6}^{d}(x)$ with the normal vector $\pm e_i$ such that $v\in \Cube_i (v)$, where $i=1,\dots,d$. Let $(u,v)$ be any diagonal of $\Gamma'$. For any $i=1,\dots,n$ let $F^{d}_{i} (v)$ be the $d$-valent subgraph of $\Gamma'$ that is the union of subgraphs $A_{a e_j} \Cube_i (u)$, $A_{a e_j} \Cube_i (v)$ and $A_{a e_j} e$, where $e$ runs over $2^{d-1}$ diagonals of $\Gamma'$ incident to $\Cube_i (v)$, $j$ runs over $\lb 1,\dots,d\rb \setminus\lb i\rb$ and $a$ runs over $\Zo$.
\end{defin}

\begin{defin}[Facets of $\Gamma$]
For any $i=0,\dots,d$ let $G_{i}(v)$ be the union of the subgraphs $F_i(v)$, $F_i(A_{q\cdot 2^{j-1}\cdot(1,\dots,1)}(v))$ and edges $E_{j}(x)$ (and their inverses), where $j$ runs over $1,\dots,r$, $q$ runs over $\Zo$ and $x$ runs over the union of all vertices of these graphs. For any $j=1,\dots,r$ define the subgraph $G_{n+j}(v)$ of $\Gamma$ to be obtained by omitting all edges $E_{j}(v)$ in $\Gamma$, where $v$ runs over $V_{\Gamma}$.
\end{defin}

\begin{prop}\label{pr:epsprop}
Let $e\in E_{\Gamma'}$ be any edge such that $\alpha(e)=\pm e_i$ holds for some $i\in \lb 1,\dots,d+1\rb$. Then one has
\begin{equation}\label{eq:chsign}
\varepsilon_{j}^{q}(i(e))=\varepsilon_{j}^{q}(t(e)),\ q\in\No;\ j=1,\dots,d+1;\ j\neq i.
\end{equation}
\end{prop}
\begin{proof}
By the definition, the values of $\varepsilon_{d+1}^{q}$ on the vertices of $I_{1/6}^{d}(x)$ are equal to each other for any $x\in L_d$. This proves \eqref{eq:chsign} for $j=d+1$ (if $i\leq d$). By definition, the values of $\varepsilon_{j}^{q}$ are mutually equal on the vertices of the graph $I_{1/6}^{d}(x)$, as well as on the vertices of any diagonal emanating from $I_{1/6}^{d}(x)$, where $j=1,\dots,d$; $q\in\No$ and $x\in\Zo^d$. Suppose that $e\in E_{I^{d}_{1/6}(x)}$ holds for some $x\in L_n\setminus\Zo^d$. Notice that $i\leq d$ holds. Without loss of generality, let $e=(i(e),i(e)+1/3e_i)$ and let $j\leq d$. Let $(v_1,i(e))$, $(v_2,t(e))$ be both unique diagonals of $\Gamma'$ terminating at $i(e)$ and $t(e)$, respectively. Notice that $A_{e_i} \Cube(v_1)=\Cube(v_2)$ holds for the respective subgraphs in $\Gamma'$. Then one uses Construction \ref{constr:sf} and \eqref{eq:epsdef} to conduct the following computation
\[
\varepsilon_{j}^{q}(t(e))=
\varepsilon_{j}^{q}(v_2)=
\varepsilon_{j}^{q}(A_{e_i}(v_1))=
\varepsilon_{j}^{q}(v_1)=
\varepsilon_{j}^{q}(i(e)),\ j\neq i;\ q\in\No.
\]
The proof is complete.
\end{proof}

\begin{lem}\label{lm:checkgkm}
The function $\alpha$ satisfies the rank, opposite and congruence conditions with respect to $\Gamma$ and  $\nabla$ (see Definition \ref{def:gkm}).
\end{lem}
\begin{proof}
Notice that the rank condition is satisfied for $\alpha$ by the construction. Let $e\in E_\Gamma$. Consider the following cases.

1) Let $e\in E_{\Gamma'}$. The opposite sign condition is easily deduced for $\alpha$ along the edge $e$. In terms of Construction \ref{constr:gammaax} of $\alpha$, one has $\alpha(e)=\pm e_q$ for some $q=1,\dots,d+1$. Then one has $w_i(i(e))= w_i(t(e))$ for any $i\neq q$ and $w_q(i(e))\equiv w_q(t(e)) \equiv 0 \pmod{e_q}$. Hence, by Proposition \ref{pr:epsprop} and by \eqref{eq:axvals} one has
\begin{multline}\label{eq:congrcomp}
\alpha(E_{j}(t(e)))=
\sum_{i=1}^{d+1} \varepsilon_{i}^{j}(t(e)) t_{i}^{j-1} w_i(t(e))\equiv
\sum_{i=1}^{d+1} \varepsilon_{i}^{j}(t(e)) t_{i}^{j-1} w_i(i(e))\equiv\\
\sum_{i=1}^{d+1} \varepsilon_{i}^{j}(i(e)) t_{i}^{j-1} w_i(i(e))=
\alpha(E_{j}(i(e))) \pmod{e_{q}},
\end{multline}
where $j=1,\dots,r$. Hence, the congruence condition holds for $\alpha$ along the edge $e$.

2) Let $e\notin E_{\Gamma'}$. Let $u=i(e)$. Then $e=E_{q}(u)$ for some $q= 1,\dots,r$. By Construction \ref{constr:gammaax}, the equality $w_i(i(e))= w_i(t(e))$ holds for any $i=1,\dots,d+1$. By \eqref{eq:opp} and \eqref{eq:axvals} then one has
\begin{multline*}
\alpha(\overline{E_{q}(i(e))})=
\alpha(E_{q}(t(e)))=
\sum_{i=1}^{d+1} \varepsilon_{i}^{q}(t(e)) t_{i}^{q-1} w_i(t(e))=\\
\sum_{i=1}^{d+1} \varepsilon_{i}^{q}(t(e)) t_{i}^{q-1} w_i(i(e))=
-\sum_{i=1}^{d+1} \varepsilon_{i}^{q}(i(e)) t_{i}^{q-1} w_i(i(e))=
-\alpha(E_{q}(i(e))).
\end{multline*}
Hence, the opposite sign condition holds for $\alpha$ along the edge $e$. If $r=1$, then the congruence relations hold for $\alpha$ along $e$. Suppose that $r\geq 2$ holds. Choose any $j=1,\dots,r$ such that $j\neq q$ holds. Then by \eqref{eq:opp} one has
\begin{multline*}
\varepsilon_{i}^{q}(t(e))=
\varepsilon_{i}^{q}(A_{2^{q-1}\cdot(1,\dots,1)}(i(e)))=
\varepsilon_{i}^{q}((A_{2^{\min\lb j,q\rb-1}\cdot(1,\dots,1)})^{2^{|j-q|}}(i(e)))=\\
(-1)^{2^{|j-q|}}\varepsilon_{i}^{q}(i(e))=
\varepsilon_{i}^{q}(i(e)),
\end{multline*}
for any $q=1,\dots,d+1$. Hence, by \eqref{eq:axvals}, the computation \eqref{eq:congrcomp} holds in this case. This implies that the congruence condition holds for $\alpha$ along the edge $e$. The proof is complete.
\end{proof}

\begin{lem}\label{lm:findvals}
Let $r \geq 1$. Then there exist integers $t_1,\dots,t_r\in\Zo$ such that the axial function $\alpha(d,r,t_1,\dots,t_r)$ is $(d+1)$-independent.
\end{lem}
\begin{proof}
For any vertex $v\in V_\Gamma$ the values of the axial function $\alpha$ on $\str(v)$ are given by the columns of the following $(d+1)\times (d+1+r)$-matrix:
\begin{equation}\label{eq:matrix}
\begin{pmatrix}
(-1)^{i_1} & \cdots & 0 & (-1)^{i_1} \varepsilon_{1}^{1}(v) & \cdots & (-1)^{i_1} \varepsilon_{1}^{r}(v) t_{1}^{r-1} \\
\vdots & \ddots & \vdots & \vdots & \ddots & \vdots\\
0 & \cdots & (-1)^{i_{d+1}} & (-1)^{i_{d+1}} \varepsilon_{d+1}^{1}(v) & \cdots & (-1)^{i_{d+1}} \varepsilon_{d+1}^{r}(v) t_{d+1}^{r-1}
\end{pmatrix},
\end{equation}
where $w_{q}(v)=(-1)^{i_q} e_q$ for $q=1,\dots,d+1$ in terms of Construction \ref{constr:gammaax}, and $i_1,\dots,i_{d+1}$ depend on $v$. By slightly abusing the notation let $M=M(v;\ j_1,\dots,j_{d+1})$ be the $(d+1)\times (d+1)$-minor of the above matrix \eqref{eq:matrix} corresponding to the columns with indices $1\leq j_1<\dots<j_{d+1}\leq d+1+r$ in \eqref{eq:matrix} (from left to right). For any integers $1\leq j_1<\dots<j_{d+1}\leq d+1+r$ there exists an integer $q\in \lb 0,\dots,d+1\rb$ such that the inequalities $j_{1},\dots,j_{q}\leq d+1$ and $j_{q+1},\dots, j_{d+1}>n+1$ hold, where $j_{0}:=0$. If $q=n+1$ then
\[
\det M= \prod_{p=1}^{d+1} (-1)^{i_{p}}.
\]
Let $q\leq n$. The ordering $t_1<\dots<t_{r}$ of the variables induces the lexicographical ordering on the polynomials from the ring $\Zo[t_1,\dots,t_{r}]$. In this ordering the maximal monomial in $\det M$ is equal to
\[
\prod_{p=1}^{q} (-1)^{j_{p}}\cdot \prod_{s=q+1}^{d+1} (-1)^{i_{s}} \varepsilon_{s}^{j_{s}}(v) t_{s}^{j_{s}-1}.
\]
In particular, $\det M$ is a nonzero polynomial in $t_1,\dots,t_r$. Hence, the left-hand sides in the system of inequalities $\det M(v;\ j_1,\dots,j_{d+1})\neq 0$, where $(j_1,\dots,j_{d+1})$ exhausts all $(d+1)$-subsets of $\lb 1,\dots,d+r+1\rb$ and $v$ runs over $V_{\Gamma}$, includes no zero polynomials. The set of real solutions for this system is the complement to the finite union of subsets of zero measure in $\Ro^r$, because $\alpha$ is periodic (see \eqref{eq:transax}). Therefore, this complement has a rational point with the corresponding coordinates $t'_1,\dots,t'_r\in\Qo$. By multiplying $t'_1,\dots,t'_r$ with the respective least common multiple one obtains $t_1,\dots,t_r\in\Zo$ such that $\alpha(n,r,t_1,\dots,t_r)$ is $(n+1)$-independent. This completes the proof.
\end{proof}

\begin{rem}
The values of the axial function $\alpha(d,r,t_1,\dots,t_r)$ obtained in Lemma \ref{lm:findvals} may not be primitive, in general. However, one can replace each non-primitive value of $\alpha(n,r,t_1,\dots,t_r)$ with the corresponding primitive vector in $\Zo^{d+1}$. Notice that the axial function obtained during this procedure is $(d+1)$-independent and its values at any star of $\Gamma$ contain a basis of $\Zo^{d+1}$.
\end{rem}

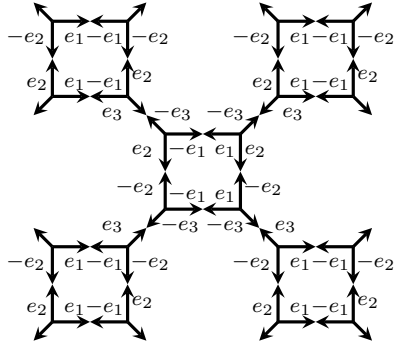
\begin{figure}
    \centering
    \input{Img_square}
    \caption{Values of the axial function $\alpha$ on $\Gamma(2,0)$}\label{fig:square}    
\end{figure}

\begin{con}[GKM-graph $\Gamma^{a}$]\label{constr:ga}
For any $a\in\Zo$ define an equivalence relation $\sim_{a}$ on $\Ro^n$ by putting $x\sim_a y$ for any $x,y\in\Ro^n$ such that $x=y+u$ for some $u\in a\cdot \Zo^d$. For any $a=b\cdot 2^r$, $b\in \Zo$, define the graph $\Gamma^{a}=\Gamma^{a}(d,r)$ to be the quotient of $\Gamma$ by $\sim_{a}$. Define the axial function $\alpha^{a}=\alpha^{a}(d,r,t_1,\dots,t_r)$ and the connection $\nabla^{a}=\nabla^{a}(d,r)$ to be induced by $\alpha$ and $\nabla$ on the quotient graph on $\Gamma^{a}$, respectively (see \eqref{eq:transax} and Fig. \ref{fig:conng2}).
\end{con}

\begin{ex}
For any $d\geq 1$ the GKM-graph $\Gamma^{1}(d,0)$ is isomorphic to the edge graph of the standard $(d+1)$-dimensional cube $\Io^{d+1}_{1}(0)$ with the axial function induced by the embedding of $\Io^{d+1}_{1}(0)$ to $\Ro^{d+1}$. 
\end{ex}

By slightly abusing the notation let $[-]=[-]_{a}\colon \Ro^d\to \Ro^d/\sim_a$ be the quotient map.

\begin{prop}
For any $a=b\cdot 2^r$, $b\in \Zo$, the GKM-graph $\Gamma^{a}(d,r)$ is a well defined graph with finitely many vertices and edges. Furthermore, it has neither multiple edges nor loops, and has type $(d+1+r,d+1)$. The objects $[A_x]$ where $x\in L_{d}$; $[\varepsilon_{i}^{j}]$ where $j=1,\dots,r$; $[F^{n}_{i}(v)]$ are well defined for any $b\in\Zo$.
\end{prop}
\begin{proof}
The quotient $\Ro^d/\sim_a$ is obtained by gluing all pairs of opposite facets of the cube $\Io^{n}_{a/2}(0)$ by respective translations in $\Ro^d$. Hence, $V_{\Gamma^a}$ is identified with $[V_{\Gamma}\cap \Io^{d}_{a/2}(0)]$. It follows from \eqref{eq:opp} that $[\varepsilon_{i}^{j}]$ is well defined for any $i=1,\dots,d+1$; $j=1,\dots,r$. The graph $\Gamma^{a}$ has neither loops nor multiple edges because the integral distance between distinct vertices of any its edge (with respect to $\Zo^d$) is less or equal to $2^{r-1}$ and the integral length of any nonzero element from $a\cdot\Zo^{d}$ is greater or equal to $2^r$. Clearly, the automorphism $[A_x]$ of $\Gamma^a$ is well-defined for any $x\in L_{d}$. We check that for any $v\in V_{\Gamma}$ the edges $[E_{i}(v)]$, $i=1,\dots,r$, are distinct. By applying $[A_{x}$] for some $x\in L_{d}$ without loss of generality assume that $v\in V_{\Cube{y}}$, where $y\in [-\frac{a}{2}(1,\dots,1),\frac{a-1}{2}\cdot (1,\dots,1)]$ belongs to the big diagonal of the respective cube. It follows from the definition that $[E_{i}(v)]$, $i=1,\dots,r$, are distinct for any $v\in V_{\Cube{y}}$ and any $y\in [-\frac{a}{2}(1,\dots,1),\frac{a-1}{2}\cdot (1,\dots,1)]$. The proof is complete.
\end{proof}

\begin{figure}
    \centering
    \input{Img_conn.tex}
    \caption{Facets of $\Gamma^{2}(2,0)$ meeting at the bottom-left vertex}\label{fig:conng2}    
\end{figure}
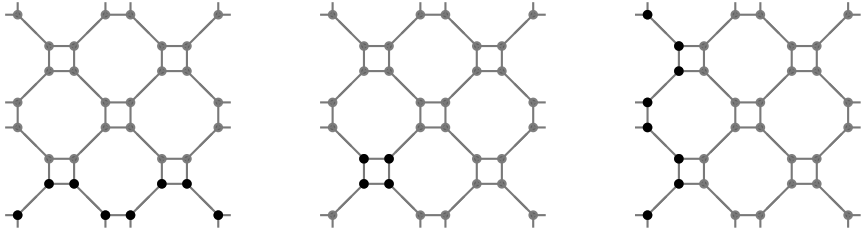

\section{Euler characteristic of face posets and proof of the main theorem}\label{sec:proof}

In this section we prove Theorem \ref{thm:example} by showing that the GKM-graph constructed in the previous section satisfies all necessary conditions. Nonextendibility is proved by Corollary \ref{cor:specface} (this would imply $(ii)$ of Theorem \ref{thm:example}). The nonrealizablity $(iii)$ is proved by using the acyclicity Theorem \ref{thm:mp} from \cite{ma-pa-06} by comparison results (Proposition \ref{pr:relposets}) and by some computations of Euler characteristic for posets given below.

Let $a\in \No$. Denote by $\pi_{J}\colon \Ro^d\to \Ro^{d}/\Ro\la e_{j}|\ j\in J\ra$ the natural projection for any $J\subseteq [n]:=\lb 1,\dots,d\rb$. The following lemma is easily proved.

\begin{lem}\label{lm:faceiso}
\begin{enumerate}[label=(\roman*)]
\item Let $\Xi$ be a $(q+1)$-face of $\Gamma^{a}(d,0)$, $q>0$. Let $\alpha\la \Xi\ra=\Zo\la e_{j},e_{d+1}|\ j\in J\ra$ for some $J\subseteq [d]$. Then $[\pi_{J}]$ induces a well-defined isomorphism $\Xi\to \Gamma^{a}(q,0)$ of $(q+1)$-valent graphs with faces (with respect to $\nabla^{a}(d,0)$). In particular, there is an isomorphism $(S_{\Gamma^{a}(d,0)})_{\leq \Xi}\cong S_{\Gamma^{a}(q,0)}$ of posets.

\item The map $\pi_{[d]\setminus J}$ induces a bijection between all $(q+1)$-faces of $\Gamma^{a}(d,0)$ with span $\Zo\la e_{j},e_{d+1}|\ j\in J\ra$ and all diagonals of $\Gamma^{a}(d-q,0)$.
\end{enumerate}
\end{lem}

\begin{lem}\label{lm:facenum}
\begin{enumerate}[label=(\roman*)]
\item The number of distinct subgraphs $\lb [\Cube(v)]|\ v\in V_{\Gamma^{a}(d,0)}\rb$ in $\Gamma^{a}(d,0)$ is equal to $2a^d$.

\item The number of $q$-faces in $\Gamma^{a}(d,0)$ is equal to
\[
2^{d-q+1}\cdot\biggl( a^d\binom{d}{q}+a^{d-q+1}\binom{d}{q-1}\biggr).
\]
\end{enumerate}
\end{lem}
\begin{proof}
The sets $\lb [x]|\ x\in \Zo^{d}\cap \relint\Io_{a/2}^{I}(0)\rb$ and $\lb [x]|\ x\in L_{n}\cap \relint\Io^{d}_{a/2}(0)\rb$ have $(a-1)^{|I|}$ and $a^d +(a-1)^d$ elements for any $I\subseteq [d]$, respectively, where $\Io_{a/2}^{I}(0)$ is the $|I|$-dimensional cube in $\Ro\la e_i|\ i\in I\ra$ with the center at $0$. By comparing $[\partial \Io^{d}_{a/2}(0)]$ and
\[
0\sqcup\bigsqcup_{q=1}^{d-1} \bigsqcup_{I\subset[d]\colon |I|=q}[\relint\Io_{a/2}^{I}(0)],
\]
one computes the cardinality in $(i)$ to be equal to
\[
a^d+(a-1)^d+\sum_{i=1}^{d-1}\binom{d}{i}(a-1)^i+1=2a^d.
\]
(Notice the identifications with respect to $\sim$, see Fig. \ref{fig:conng2}.) This proves $(i)$. Recall that the number of $q$-faces of a $d$-cube is equal to $2^{d-q}\binom{q}{d}$. Notice that any $q$-face of $(\Gamma^{a}(d,0),\nabla^{a}(d,0))$ is either a $q$-face of $\Cube(v)$ for some $v\in V_{\Gamma^{a}(d,0)}$ or has a diagonal. The number of all $q$-faces in $\Gamma^a(d,0)$ with no diagonal is equal to $2a^{d}\cdot 2^{d-q}\binom{d}{q}$ by $(i)$. The set of diagonals in $\Gamma^{a}(d,0)$ is in bijection with \[
\bigsqcup_{x\in (L_d\setminus \Zo^d)\cap \relint \Io_{a/2}^{d}(0)} V_{I_{1/6}(x)}.
\]
Hence, the number of diagonals in $\Gamma^{a}(d,0)$ is equal to $(2a)^d$. Applying $(i)$ and Lemma \ref{lm:faceiso} $(ii)$ one computes the number of all $q$-faces having a diagonal in $\Gamma^{a}(d,0)$ to be equal to $2^{d-q+1} a^{d-q+1}\binom{d}{q-1}$. By summing up these two cardinalities one obtains the desired expression from $(ii)$. The proof is complete.
\end{proof}

For any $d\geq 1$ and $r\geq 0$ let $\Gamma^{d+1+r}_{d+1}:=\Gamma^{2^{r+1}}(d,r)$ be the $(d+1+r)$-valent GKM-graph of rank $d+1$. 

\begin{rem}
We use $\Gamma^{2^{r+1}}(d,r)$ instead of the GKM-graph $\Gamma^{2^{r}}(d,r)$ with a smaller number of vertices in order to include the graph with correct properties (see Lemma \ref{lm:eucomp}) for $r=0$. For $r>0$ one can take $\Gamma^{2^{r}}(d,r)$ in the proof below.
\end{rem}

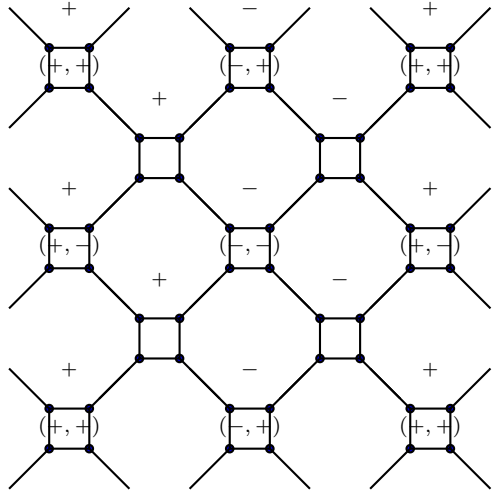
\begin{figure}
  \input{Img_cross}
  \caption{The graph $\Gamma^{2}(2,1)$ (chords are not depicted) and some values of $\varepsilon_{i}^{q}$. For any $v\in\Zo^2$, the signs of $\varepsilon_{1}^{1},\varepsilon_{2}^{1}$ are constant at $V_{[\Cube(v)]_2}$ with fixed $v$. The signs of $\varepsilon_{3}^{1}$ on the vertices of any square are shown above the respective square}\label{fig:cross}
\end{figure}
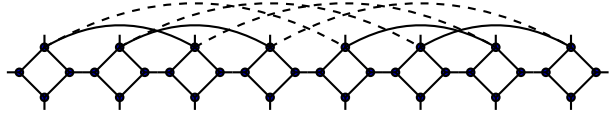
\begin{figure}
  \input{Img_edges}
  \caption{Some chords (curvilinear) of $\Gamma'$ in $\Gamma^{4}(2,2)$ in the row of squares along the direction $(1,1)\in\Ro^2$}\label{fig:edges}
\end{figure}

\begin{lem}\label{lm:eucomp}
For any $a\in \No$ one has
\begin{equation}\label{lm:eucomp1}
\tilde{\chi} \bigl(\Delta (\overline{S_{\Gamma^{a}(d,0)}^{op}})\bigr)=
(-1)^d\bigl(2a^d-(2a-1)^d\bigr).
\end{equation}
In particular, $\tilde{\chi} \bigl(\Delta (\overline{S_{\Gamma^{a}(d,0)}^{op}})\bigr)$ is nonzero and has sign $(-1)^{d+1}$ for any $a,d\geq 2$.
\end{lem}
\begin{proof}
Since the connection on $\Gamma^{a}(d,0)$ is complete one may apply Proposition \ref{pr:euch} by Proposition \ref{pr:condgenpos} to $\Delta (\overline{S_{\Gamma^{a}(d,0)}^{op}})$. The computation of the corresponding Euler characteristic is then given as follows by using Lemma \ref{lm:facenum}:
\begin{multline*}
\tilde{\chi} \bigl(\Delta (\overline{S_{\Gamma^{a}(d,0)}^{op}})\bigr)=
-1+\sum_{q=0}^{d} (-1)^{q} 2^{q+1}\biggl( a^d \binom{n}{d-q}+a^{q+1}\binom{d}{d-q-1}\biggr)=\\
-1+2a^d \sum_{q=0}^{d} (-1)^{q} \binom{d}{q}2^{q} + \sum_{q=0}^{d} (-1)^{q}\binom{d}{q+1}(2a)^{q+1}=
-1+2a^d(-1)^d+1-(1-2a)^d=\\
(-1)^d\bigl(2a^d-(2a-1)^d\bigr).
\end{multline*}
The claim about the sign follows from the obvious inequality which holds for any $a,d\geq 2$:
\[
2<\biggl(\frac{2a-1}{a}\biggr)^d.
\]
\end{proof}

\begin{ex}\label{ex:sphere}
Notice that if $a=1$ or $d=1$ then $\Delta (\overline{S_{\Gamma^{a}(d,r)}^{op}})$ is homeomorphic to a $d$-sphere. In this case Lemma \ref{lm:eucomp} implies that its reduced Euler characteristic is equal to $(-1)^d$.
\end{ex}

Now everything is at hand to prove the main theorem of this paper.

\begin{thm}\label{thm:example_}
For any $n\geq k\geq 3$ there exists an $(n,k)$-type GKM-graph $\Gamma$ satisfying the following properties:
\begin{enumerate}[label=(\roman*)]
\item The GKM-graph $\Gamma$ is $k$-independent;
\item The GKM-graph $\Gamma$ does not have nontrivial extensions;
\item The GKM-graph $\Gamma$ is not realized by a GKM-action if $n=k=3$ or $n\geq k\geq 4$.
\end{enumerate}
\end{thm}
\begin{proof}
Lemma \ref{lm:checkgkm} implies that $\Gamma^{d+1+r}_{d+1}$ is a GKM-graph. We check that the properties $(i)$, $(ii)$ hold for such a GKM-graph. The existence of $t_1,\dots,t_r$ such that $(i)$ holds is granted by Lemma \ref{lm:findvals}. By the construction, the $(d+1)$-face $\Xi:=\Gamma^{2^{r+1}}(d,0)$ of $\Gamma^{d+1+r}_{d+1}$ has $r$ chords $E_{j}(v)$ at any vertex $v\in V_{\Xi}$, where $j$ runs over $1,\dots,r$. In particular, any transversal edge to the $(d+1)$-face $\Xi$ in $\Gamma$ is a chord for the face $\Xi$. Hence, by Corollary \ref{cor:specface} the GKM-graph $\Gamma^{d+1+r}_{d+1}$ has no nontrivial extensions. This establishes $(i)$, $(ii)$.

Now we prove the nonrealizability $(iii)$ by assuming the contrary and obtaining a contradiction with acyclicity properties given by Theorem \ref{thm:mp} from \cite{ma-pa-06}. Let $d\geq 3$. Suppose that there exists a GKM-action of $T^{d+1}$ on $M^{2(d+1+r)}$ yielding the GKM-graph $\Gamma^{d+1+r}_{d+1}$. Consider a face $\Omega$ that is isomorphic to $\Gamma^{2^{r+1}}(d-1,0)$ in $\Gamma^{d+1+r}_{d+1}$. By Proposition \ref{pr:relposets} the face $\Omega$ is realizable by a torus (face) manifold $F$ in $M$, and $(S_{M})_{\leq F}$ is isomorphic to $(S_{\Gamma^{d+1+r}_{d+1}})_{\leq \Omega}$. By Theorem \ref{thm:mp}, the poset $\overline{(S_{M})_{\leq F}^{op}}$ is $(d-2)$-acyclic. Hence, one has
\[
\tilde{\chi} \bigl(\Delta((S_{\Gamma^{d+1+r}_{d+1}})_{<\Omega})\bigr)=
\tilde{\chi}\bigl(\Delta(\overline{(S_{M})_{\leq F}^{op}})\bigr)=
(-1)^{d-1} \rk H^{d-1}(\overline{(S_{M})_{\leq F}^{op}}).
\]
By Lemma \ref{lm:faceiso} one has
\[
\tilde{\chi} \bigl(\Delta((S_{\Gamma^{d+1+r}_{d+1}})_{<\Omega})\bigr)=
\tilde{\chi} \bigl(\Delta(\overline{(S_{\Gamma^{a}(d-1,0)})^{op}})\bigr).
\]
However, by Lemma \ref{lm:eucomp} the last expression is nonzero and has sign $(-1)^{d}$. This contradiction proves $(iv)$. For $d=2$ and $r=0$ the proof is conducted in a similar way to the above by taking $\Omega=\Gamma^{3}_{3}$. The proof is complete.
\end{proof}

\begin{rem}
Clearly, $3$-independence property ($n>k\geq 3$, or equivalently $d=2$, $r> 0$) is not enough to prove nonrealizability of $\Gamma^{d+1+r}_{d+1}$ by using the argument from the proof above. A possible strategy to prove nonrealizability of $\Gamma^{d+1+r}_{d+1}$ for $d=2$ is to show that the poset $(S_{\Gamma^{d+1+r}_{d+1}})_{2}$ consisting of faces of dimension not exceeding $2$ has nontrivial first homology. Indeed, by the acyclicity theorem of \cite{ay-ma-so-22}, the poset $(S_{M})_{2}$ of faces of dimension not exceeding $2$ in $S_M$ is $1$-acyclic. On the other hand, the posets $(S_{\Gamma^{d+1+r}_{d+1}})_{2}$ and $(S_{M})_{2}$ can be shown to be isomorphic (in the $3$-independent case) by using Proposition \ref{pr:relposets}. However, we do not pursue such a tedious task here.
\end{rem}

\begin{rem}
Recall that a poset $P$ is called a \emph{homology manifold} if the order complex $\Delta(P)$ is a homology manifold in the usual sense. Notice that the poset $\overline{S_{\Gamma^{a}(d,0)}^{op}}$ is not a homology manifold for any $d\geq 3$ and any $a\geq 2$. Indeed, let $\Omega\in \overline{S_{\Gamma^{a}(d,0)}^{op}}$ be as above. Then the poset $\overline{(S_{\Gamma^{a}(d,0)})_{\leq\Omega}^{op}}$ is isomorphic to $\overline{S_{\Gamma^{a}(d-1,0)}^{op}}$ by Lemma \ref{lm:faceiso}. The order complex of the latter poset is not a homology sphere by inspecting the respective Euler characteristic from Lemma \ref{lm:eucomp}. Hence, the link of the one-element chain $G$ in $\Delta (\overline{S_{\Gamma^{a}(d,0)}^{op}})$ is not a homology sphere, as required.
\end{rem}

\begin{rem}
Define the graph $G=G(d)$ embedded into $\Ro^{d}$ as the union of edges
\[
F(x,u):=(x,\ x+\frac{1}{2}\cdot \sum_{i=1}^d (-1)^{u_i} e_i),
\]
where $x$ runs over $\Zo^d$, $u=(u_1,\dots,u_d)$ runs over $\lb \pm 1\rb^d$. Let $G^{a}(d):=G(d)/\sim_{a}$ be the $2^d$-regular graph $G^{a}(d)$ embedded to the torus $(S^1)^d$ for any $a\geq 1$. Let $p\colon \Gamma^{a}(d,0)\to G$ be the graph morphism given by replacing every cubical face with the respective vertex and mapping $D(x,u)$ to $F(x,u)$. A face in the obtained graph is by definition a subgraph in $G^{a}(d)$ obtained as an image of a face under $p$. Define the poset $P(d,a)$ as the poset of all faces in $G^{a}(d)$. One has $\dim P(d,a)=d$. Clearly, the order complexes of $\overline{P(2,a)^{op}}$ and of $\overline{S_{\Gamma^{a}(d,0)}^{op}}$ are homotopy equivalent. Conjecturally, for any $a\geq 2$ one has a homotopy equivalence
\[
\Delta(\overline{P(2,a)^{op}})\simeq (S^1)^{\vee 4a-1}\vee (S^2)^{\vee 2a^2}.
\]
The resulting Euler characteristic agrees with the formula from Lemma \ref{lm:eucomp}.
\end{rem}

\section*{Acknowledgements}
The author would like to thank Shintar\^o Kuroki for many fruitful discussions of GKM-theory, as well as A.Ayzenberg for some general remarks about the text and for some useful observations on geometry of torus actions.

\end{document}

%% file: Img_square.tex
\begin{tikzpicture}[
	scale=1.500000,
	back/.style={loosely dotted, thick},
	edge/.style={color=black!95!black, thick},
	facet/.style={fill=gray!95!black,fill opacity=0.300000},
	facet0/.style={fill=orange!95!black,fill opacity=0.300000},
	vertex/.style={inner sep=1pt,circle,draw=black!25!black,fill=blue!75!black,thick,anchor=base},
	arr/.style={->,>=stealth, black, very thick}
	]
	\tikzstyle{every node}=[font=\tiny]
    \foreach \x in {0,2} {
        \foreach \y in {0,2} {
            %labels
            \node at (\x-1/3,\y-1/3) [label={[shift={(0.3,-0.2)}]$e_1$}]{};
            \node at (\x-1/3,\y-1/3) [label={[shift={(-0.2,-0.2)}]$e_2$}]{};
            
            \node at (\x+1/3,\y-1/3) [label={[shift={(0.2,-0.1)}]$e_2$}]{};
            \node at (\x+1/3,\y-1/3) [label={[shift={(-0.3,-0.2)}]$-e_1$}]{};

            \node at (\x-1/3,\y+1/3) [label={[shift={(-0.35,-0.6)}]$-e_2$}]{};
            \node at (\x-1/3,\y+1/3) [label={[shift={(0.3,-0.6)}]$e_1$}]{};

            \node at (\x+1/3,\y+1/3) [label={[shift={(-0.3,-0.6)}]$-e_1$}]{};
            \node at (\x+1/3,\y+1/3) [label={[shift={(0.3,-0.6)}]$-e_2$}]{};
            
            %arrows
            \draw[arr] (\x-1/3,\y-1/3) -- (\x,\y-1/3);
            \draw[arr] (\x-1/3,\y-1/3) -- (\x-1/3,\y);

            \draw[arr] (\x+1/3,\y-1/3) -- (\x,\y-1/3);
            \draw[arr] (\x+1/3,\y-1/3) -- (\x+1/3,\y);
        
            \draw[arr] (\x-1/3,\y+1/3) -- (\x-1/3,\y);
            \draw[arr] (\x-1/3,\y+1/3) -- (\x,\y+1/3);
        
            \draw[arr] (\x+1/3,\y+1/3) -- (\x,\y+1/3);
            \draw[arr] (\x+1/3,\y+1/3) -- (\x+1/3,\y);                    
        }            
    }        
    %labels
    \node at (1-1/3,1-1/3) [label={[shift={(0.3,-0.2)}]$-e_1$}]{};
    \node at (1-1/3,1-1/3) [label={[shift={(-0.4,-0.1)}]$-e_2$}]{};
    \node at (1-1/3,1-1/3) [label={[shift={(0.2,-0.6)}]$-e_3$}]{};
    
    \node at (1+1/3,1-1/3) [label={[shift={(0.3,-0.1)}]$-e_2$}]{};
    \node at (1+1/3,1-1/3) [label={[shift={(-0.2,-0.2)}]$e_1$}]{};
    \node at (1+1/3,1-1/3) [label={[shift={(-0.2,-0.6)}]$-e_3$}]{};
    
    \node at (1-1/3,1+1/3) [label={[shift={(0.1,-0.1)}]$-e_3$}]{};
    \node at (1-1/3,1+1/3) [label={[shift={(-0.3,-0.6)}]$e_2$}]{};
    \node at (1-1/3,1+1/3) [label={[shift={(0.3,-0.6)}]$-e_1$}]{};
    
    \node at (1+1/3,1+1/3) [label={[shift={(-0.2,-0.1)}]$-e_3$}]{};
    \node at (1+1/3,1+1/3) [label={[shift={(-0.2,-0.6)}]$e_1$}]{};
    \node at (1+1/3,1+1/3) [label={[shift={(0.2,-0.6)}]$e_2$}]{};

    \node at (2-1/3,2-1/3) [label={[shift={(0.2,-0.6)}]$e_3$}]{};

    \node at (0+1/3,2-1/3) [label={[shift={(-0.2,-0.6)}]$e_3$}]{};

    \node at (2-1/3,0+1/3) [label={[shift={(0.1,-0.1)}]$e_3$}]{};

    \node at (0+1/3,0+1/3) [label={[shift={(-0.2,-0.1)}]$e_3$}]{};
    
    %arrows
    \draw[arr] (1-1/3,1-1/3) -- (1,1-1/3);
    \draw[arr] (1-1/3,1-1/3) -- (1-1/3,1);
    \draw[arr] (1-1/3,1-1/3) -- (1-1/2,1-1/2);

    \draw[arr] (1+1/3,1-1/3) -- (1,1-1/3);
    \draw[arr] (1+1/3,1-1/3) -- (1+1/3,1);
    \draw[arr] (1+1/3,1-1/3) -- (1+1/2,1-1/2);

    \draw[arr] (1-1/3,1+1/3) -- (1-1/3,1);
    \draw[arr] (1-1/3,1+1/3) -- (1,1+1/3);
    \draw[arr] (1-1/3,1+1/3) -- (1-1/2,1+1/2);

    \draw[arr] (1+1/3,1+1/3) -- (1,1+1/3);
    \draw[arr] (1+1/3,1+1/3) -- (1+1/3,1);
    \draw[arr] (1+1/3,1+1/3) -- (1+1/2,1+1/2);

    %external arrows
    \foreach \x in {0,2} {
        \foreach \y in {0,2} {
             \foreach \u in {-1/3,1/3} {
                 \foreach \v in {-1/3,1/3} {
                    \draw[arr] (\x+\u,\y+\v) -- (\x+3/2*\u,\y+3/2*\v);
                 }
            }
        }
    }            
\end{tikzpicture}	

%% file: Img_conn.tex
    \begin{subfigure}[b]{0.25\textwidth}
        \centering
        \begin{tikzpicture}[
        	scale=0.50000,
        	back/.style={loosely dotted, thick},
        	edge/.style={color=black!95!black, thick},
        	edgeg/.style={color=gray!95!black, thick},
        	vertex/.style={inner sep=1pt,circle,draw=black!25!black,fill=black!75!black,thick,anchor=base},
        	vertexg/.style={inner sep=1pt,circle,draw=gray!25!gray,fill=gray!75!black,thick,anchor=base},
        ]

        %central integral
        \foreach \u in {-1/3,1/3} {
            \foreach \v in {-1/3,1/3} {
                \node[vertexg] at (3+\u,3+\v) {};
                %edges between squares
                \draw[edgeg] (3+\u,3+\v) -- (3+3*\u,3+3*\v);
            }
        }
        \draw[edgeg] (3-1/3,3-1/3) -- (3+1/3,3-1/3) -- (3+1/3,3+1/3) -- (3-1/3,3+1/3) -- (3-1/3,3-1/3);

        %left-upper
        \node[vertexg] at (1/3,6-1/3) {};
        %edges between squares
        \draw[edgeg] (1/3,6-1/3) -- (0,6-1/3);
        \draw[edgeg] (1/3,6-1/3) -- (1/3,6-1/3+1/3);
        %diagonals
        \draw[edgeg] (1/3,6-1/3) -- (1/3+1/2,6-1/3-1/2);

        %left-down
        \node[vertexg] at (1/3,1/3) {};
        %edges between squares
        \draw[edgeg] (1/3,1/3) -- (0,1/3);
        \draw[edgeg] (1/3,1/3) -- (1/3,0);
        %diagonals
        \draw[edgeg] (1/3,1/3) -- (1/3+1/2,1/3+1/2);

        %right-down
        \node[vertexg] at (6-1/3,1/3) {};
        %edges between squares
        \draw[edgeg] (6-1/3,1/3) -- (6,1/3);
        \draw[edgeg] (6-1/3,1/3) -- (6-1/3,0);
        %diagonals
        \draw[edgeg] (6-1/3,1/3) -- (6-1/3-1/2,1/3+1/2);

        %right-upper
        \node[vertexg] at (6-1/3,6-1/3) {};
        %edges between squares
        \draw[edgeg] (6-1/3,6-1/3) -- (6,6-1/3);
        \draw[edgeg] (6-1/3,6-1/3) -- (6-1/3,6);
        %diagonals
        \draw[edgeg] (6-1/3,6-1/3) -- (6-1/3-1/2,6-1/3-1/2);

        %left
        \node[vertexg] at (1/3,3-1/3) {};
        \node[vertexg] at (1/3,3+1/3) {};
        %edges between squares
        \draw[edgeg] (1/3,3-1/3) -- (0,3-1/3);
        \draw[edgeg] (1/3,3+1/3) -- (0,3+1/3);
        \draw[edgeg] (1/3,3-1/3) -- (1/3,3+1/3);
        %diagonals
        \draw[edgeg] (1/3,3-1/3) -- (1/3+1/2,3-1/3-1/2);
        \draw[edgeg] (1/3,3+1/3) -- (1/3+1/2,3+1/3+1/2);

        %down
        \node[vertexg] at (3-1/3,1/3) {};
        \node[vertexg] at (3+1/3,1/3) {};
        %edges between squares
        \draw[edgeg] (3-1/3,1/3) -- (3-1/3,0);
        \draw[edgeg] (3+1/3,1/3) -- (3+1/3,0);
        \draw[edgeg] (3-1/3,1/3) -- (3+1/3,1/3);
        %diagonals
        \draw[edgeg] (3-1/3,1/3) -- (3-1/3-1/2,1/3+1/2);
        \draw[edgeg] (3+1/3,1/3) -- (3+1/3+1/2,1/3+1/2);

        %right
        \node[vertexg] at (6-1/3,3-1/3) {};
        \node[vertexg] at (6-1/3,3+1/3) {};
        %edges between squares
        \draw[edgeg] (6-1/3,3-1/3) -- (6,3-1/3);
        \draw[edgeg] (6-1/3,3+1/3) -- (6,3+1/3);
        \draw[edgeg] (6-1/3,3-1/3) -- (6-1/3,3+1/3);
        %diagonals
        \draw[edgeg] (6-1/3,3-1/3) -- (6-1/3-1/2,3-1/3-1/2);
        \draw[edgeg] (6-1/3,3+1/3) -- (6-1/3-1/2,3+1/3+1/2);

        %upper
        \node[vertexg] at (3-1/3,6-1/3) {};
        \node[vertexg] at (3+1/3,6-1/3) {};
        %edges between squares
        \draw[edgeg] (3-1/3,6-1/3) -- (3-1/3,6);
        \draw[edgeg] (3+1/3,6-1/3) -- (3+1/3,6);
        \draw[edgeg] (3-1/3,6-1/3) -- (3+1/3,6-1/3);
        %diagonals
        \draw[edgeg] (3-1/3,6-1/3) -- (3-1/3-1/2,6-1/3-1/2);
        \draw[edgeg] (3+1/3,6-1/3) -- (3+1/3+1/2,6-1/3-1/2);

        \foreach \x in {0,...,1} {
            \foreach \y in {0,...,1} {
                \foreach \u in {-1/3,1/3} {
                    \foreach \v in {-1/3,1/3} {
                    %vertices
                        \node[vertexg] at (3*\x+3/2+\u,3*\y+3/2+\v) {};
                        %edges between squares
                        \draw[edgeg] (3*\x+3/2+\u,3*\y+3/2+\v) -- (3*\x+3/2+3*\u,3*\y+3/2+3*\v);
                    }
                }
                %edges of squares
                \draw[edgeg] (3*\x+3/2-1/3,3*\y+3/2-1/3) -- (3*\x+3/2+1/3,3*\y+3/2-1/3) -- (3*\x+3/2+1/3,3*\y+3/2+1/3) -- (3*\x+3/2-1/3,3*\y+3/2+1/3) -- (3*\x+3/2-1/3,3*\y+3/2-1/3);
            }
        }

            %path
            \node[vertex] at (1/3, 1/3) {};
            \node[vertex] at (3/2-1/3, 3/2-1/3) {};
            \node[vertex] at (3/2+1/3, 3/2-1/3) {};
            \node[vertex] at (3-1/3, 1/3) {};
            \node[vertex] at (3+1/3, 1/3) {};
            \node[vertex] at (3+3/2-1/3, 3/2-1/3) {};
            \node[vertex] at (3+3/2+1/3, 3/2-1/3) {};
            \node[vertex] at (6-1/3, 1/3) {};

        \end{tikzpicture}
    \end{subfigure}
    \begin{subfigure}[b]{0.25\textwidth}
        \centering
        \begin{tikzpicture}[
        	scale=0.50000,
        	back/.style={loosely dotted, thick},
        	edge/.style={color=black!95!black, thick},
        	edgeg/.style={color=gray!95!black, thick},
        	vertex/.style={inner sep=1pt,circle,draw=black!25!black,fill=black!75!black,thick,anchor=base},
        	vertexg/.style={inner sep=1pt,circle,draw=gray!25!gray,fill=gray!75!black,thick,anchor=base},
        ]

        %central integral
        \foreach \u in {-1/3,1/3} {
            \foreach \v in {-1/3,1/3} {
                \node[vertexg] at (3+\u,3+\v) {};
                %edges between squares
                \draw[edgeg] (3+\u,3+\v) -- (3+3*\u,3+3*\v);
            }
        }
        \draw[edgeg] (3-1/3,3-1/3) -- (3+1/3,3-1/3) -- (3+1/3,3+1/3) -- (3-1/3,3+1/3) -- (3-1/3,3-1/3);

        %left-upper
        \node[vertexg] at (1/3,6-1/3) {};
        %edges between squares
        \draw[edgeg] (1/3,6-1/3) -- (0,6-1/3);
        \draw[edgeg] (1/3,6-1/3) -- (1/3,6-1/3+1/3);
        %diagonals
        \draw[edgeg] (1/3,6-1/3) -- (1/3+1/2,6-1/3-1/2);

        %left-down
        \node[vertexg] at (1/3,1/3) {};
        %edges between squares
        \draw[edgeg] (1/3,1/3) -- (0,1/3);
        \draw[edgeg] (1/3,1/3) -- (1/3,0);
        %diagonals
        \draw[edgeg] (1/3,1/3) -- (1/3+1/2,1/3+1/2);

        %right-down
        \node[vertexg] at (6-1/3,1/3) {};
        %edges between squares
        \draw[edgeg] (6-1/3,1/3) -- (6,1/3);
        \draw[edgeg] (6-1/3,1/3) -- (6-1/3,0);
        %diagonals
        \draw[edgeg] (6-1/3,1/3) -- (6-1/3-1/2,1/3+1/2);

        %right-upper
        \node[vertexg] at (6-1/3,6-1/3) {};
        %edges between squares
        \draw[edgeg] (6-1/3,6-1/3) -- (6,6-1/3);
        \draw[edgeg] (6-1/3,6-1/3) -- (6-1/3,6);
        %diagonals
        \draw[edgeg] (6-1/3,6-1/3) -- (6-1/3-1/2,6-1/3-1/2);

        %left
        \node[vertexg] at (1/3,3-1/3) {};
        \node[vertexg] at (1/3,3+1/3) {};
        %edges between squares
        \draw[edgeg] (1/3,3-1/3) -- (0,3-1/3);
        \draw[edgeg] (1/3,3+1/3) -- (0,3+1/3);
        \draw[edgeg] (1/3,3-1/3) -- (1/3,3+1/3);
        %diagonals
        \draw[edgeg] (1/3,3-1/3) -- (1/3+1/2,3-1/3-1/2);
        \draw[edgeg] (1/3,3+1/3) -- (1/3+1/2,3+1/3+1/2);

        %down
        \node[vertexg] at (3-1/3,1/3) {};
        \node[vertexg] at (3+1/3,1/3) {};
        %edges between squares
        \draw[edgeg] (3-1/3,1/3) -- (3-1/3,0);
        \draw[edgeg] (3+1/3,1/3) -- (3+1/3,0);
        \draw[edgeg] (3-1/3,1/3) -- (3+1/3,1/3);
        %diagonals
        \draw[edgeg] (3-1/3,1/3) -- (3-1/3-1/2,1/3+1/2);
        \draw[edgeg] (3+1/3,1/3) -- (3+1/3+1/2,1/3+1/2);

        %right
        \node[vertexg] at (6-1/3,3-1/3) {};
        \node[vertexg] at (6-1/3,3+1/3) {};
        %edges between squares
        \draw[edgeg] (6-1/3,3-1/3) -- (6,3-1/3);
        \draw[edgeg] (6-1/3,3+1/3) -- (6,3+1/3);
        \draw[edgeg] (6-1/3,3-1/3) -- (6-1/3,3+1/3);
        %diagonals
        \draw[edgeg] (6-1/3,3-1/3) -- (6-1/3-1/2,3-1/3-1/2);
        \draw[edgeg] (6-1/3,3+1/3) -- (6-1/3-1/2,3+1/3+1/2);

        %upper
        \node[vertexg] at (3-1/3,6-1/3) {};
        \node[vertexg] at (3+1/3,6-1/3) {};
        %edges between squares
        \draw[edgeg] (3-1/3,6-1/3) -- (3-1/3,6);
        \draw[edgeg] (3+1/3,6-1/3) -- (3+1/3,6);
        \draw[edgeg] (3-1/3,6-1/3) -- (3+1/3,6-1/3);
        %diagonals
        \draw[edgeg] (3-1/3,6-1/3) -- (3-1/3-1/2,6-1/3-1/2);
        \draw[edgeg] (3+1/3,6-1/3) -- (3+1/3+1/2,6-1/3-1/2);

        \foreach \x in {0,...,1} {
            \foreach \y in {0,...,1} {
                \foreach \u in {-1/3,1/3} {
                    \foreach \v in {-1/3,1/3} {
                    %vertices
                        \node[vertexg] at (3*\x+3/2+\u,3*\y+3/2+\v) {};
                        %edges between squares
                        \draw[edgeg] (3*\x+3/2+\u,3*\y+3/2+\v) -- (3*\x+3/2+3*\u,3*\y+3/2+3*\v);
                    }
                }
                %edges of squares
                \draw[edgeg] (3*\x+3/2-1/3,3*\y+3/2-1/3) -- (3*\x+3/2+1/3,3*\y+3/2-1/3) -- (3*\x+3/2+1/3,3*\y+3/2+1/3) -- (3*\x+3/2-1/3,3*\y+3/2+1/3) -- (3*\x+3/2-1/3,3*\y+3/2-1/3);
            }
        }

            %path
         %   \node[vertex] at (1/3, 1/3) {};
        \node[vertex] at (3/2-1/3, 3/2-1/3) {};
        \node[vertex] at (3/2+1/3, 3/2-1/3) {};
        \node[vertex] at (3/2-1/3, 3/2+1/3) {};
        \node[vertex] at (3/2+1/3, 3/2+1/3) {};
        % \node[vertex] at (3/2-1/3, 3/2+1/3) {};
        % \node[vertex] at (1/3, 3-1/3) {};

        \end{tikzpicture}
    \end{subfigure}
    \begin{subfigure}[b]{0.25\textwidth}
        \centering
        \begin{tikzpicture}[
        	scale=0.50000,
        	back/.style={loosely dotted, thick},
        	edge/.style={color=black!95!black, thick},
        	edgeg/.style={color=gray!95!black, thick},
        	vertex/.style={inner sep=1pt,circle,draw=black!25!black,fill=black!75!black,thick,anchor=base},
        	vertexg/.style={inner sep=1pt,circle,draw=gray!25!gray,fill=gray!75!black,thick,anchor=base},
        ]

        %central integral
        \foreach \u in {-1/3,1/3} {
            \foreach \v in {-1/3,1/3} {
                \node[vertexg] at (3+\u,3+\v) {};
                %edges between squares
                \draw[edgeg] (3+\u,3+\v) -- (3+3*\u,3+3*\v);
            }
        }
        \draw[edgeg] (3-1/3,3-1/3) -- (3+1/3,3-1/3) -- (3+1/3,3+1/3) -- (3-1/3,3+1/3) -- (3-1/3,3-1/3);

        %left-upper
        \node[vertexg] at (1/3,6-1/3) {};
        %edges between squares
        \draw[edgeg] (1/3,6-1/3) -- (0,6-1/3);
        \draw[edgeg] (1/3,6-1/3) -- (1/3,6-1/3+1/3);
        %diagonals
        \draw[edgeg] (1/3,6-1/3) -- (1/3+1/2,6-1/3-1/2);

        %left-down
        \node[vertexg] at (1/3,1/3) {};
        %edges between squares
        \draw[edgeg] (1/3,1/3) -- (0,1/3);
        \draw[edgeg] (1/3,1/3) -- (1/3,0);
        %diagonals
        \draw[edgeg] (1/3,1/3) -- (1/3+1/2,1/3+1/2);

        %right-down
        \node[vertexg] at (6-1/3,1/3) {};
        %edges between squares
        \draw[edgeg] (6-1/3,1/3) -- (6,1/3);
        \draw[edgeg] (6-1/3,1/3) -- (6-1/3,0);
        %diagonals
        \draw[edgeg] (6-1/3,1/3) -- (6-1/3-1/2,1/3+1/2);

        %right-upper
        \node[vertexg] at (6-1/3,6-1/3) {};
        %edges between squares
        \draw[edgeg] (6-1/3,6-1/3) -- (6,6-1/3);
        \draw[edgeg] (6-1/3,6-1/3) -- (6-1/3,6);
        %diagonals
        \draw[edgeg] (6-1/3,6-1/3) -- (6-1/3-1/2,6-1/3-1/2);

        %left
        \node[vertexg] at (1/3,3-1/3) {};
        \node[vertexg] at (1/3,3+1/3) {};
        %edges between squares
        \draw[edgeg] (1/3,3-1/3) -- (0,3-1/3);
        \draw[edgeg] (1/3,3+1/3) -- (0,3+1/3);
        \draw[edgeg] (1/3,3-1/3) -- (1/3,3+1/3);
        %diagonals
        \draw[edgeg] (1/3,3-1/3) -- (1/3+1/2,3-1/3-1/2);
        \draw[edgeg] (1/3,3+1/3) -- (1/3+1/2,3+1/3+1/2);

        %down
        \node[vertexg] at (3-1/3,1/3) {};
        \node[vertexg] at (3+1/3,1/3) {};
        %edges between squares
        \draw[edgeg] (3-1/3,1/3) -- (3-1/3,0);
        \draw[edgeg] (3+1/3,1/3) -- (3+1/3,0);
        \draw[edgeg] (3-1/3,1/3) -- (3+1/3,1/3);
        %diagonals
        \draw[edgeg] (3-1/3,1/3) -- (3-1/3-1/2,1/3+1/2);
        \draw[edgeg] (3+1/3,1/3) -- (3+1/3+1/2,1/3+1/2);

        %right
        \node[vertexg] at (6-1/3,3-1/3) {};
        \node[vertexg] at (6-1/3,3+1/3) {};
        %edges between squares
        \draw[edgeg] (6-1/3,3-1/3) -- (6,3-1/3);
        \draw[edgeg] (6-1/3,3+1/3) -- (6,3+1/3);
        \draw[edgeg] (6-1/3,3-1/3) -- (6-1/3,3+1/3);
        %diagonals
        \draw[edgeg] (6-1/3,3-1/3) -- (6-1/3-1/2,3-1/3-1/2);
        \draw[edgeg] (6-1/3,3+1/3) -- (6-1/3-1/2,3+1/3+1/2);

        %upper
        \node[vertexg] at (3-1/3,6-1/3) {};
        \node[vertexg] at (3+1/3,6-1/3) {};
        %edges between squares
        \draw[edgeg] (3-1/3,6-1/3) -- (3-1/3,6);
        \draw[edgeg] (3+1/3,6-1/3) -- (3+1/3,6);
        \draw[edgeg] (3-1/3,6-1/3) -- (3+1/3,6-1/3);
        %diagonals
        \draw[edgeg] (3-1/3,6-1/3) -- (3-1/3-1/2,6-1/3-1/2);
        \draw[edgeg] (3+1/3,6-1/3) -- (3+1/3+1/2,6-1/3-1/2);

        \foreach \x in {0,...,1} {
            \foreach \y in {0,...,1} {
                \foreach \u in {-1/3,1/3} {
                    \foreach \v in {-1/3,1/3} {
                    %vertices
                        \node[vertexg] at (3*\x+3/2+\u,3*\y+3/2+\v) {};
                        %edges between squares
                        \draw[edgeg] (3*\x+3/2+\u,3*\y+3/2+\v) -- (3*\x+3/2+3*\u,3*\y+3/2+3*\v);
                    }
                }
                %edges of squares
                \draw[edgeg] (3*\x+3/2-1/3,3*\y+3/2-1/3) -- (3*\x+3/2+1/3,3*\y+3/2-1/3) -- (3*\x+3/2+1/3,3*\y+3/2+1/3) -- (3*\x+3/2-1/3,3*\y+3/2+1/3) -- (3*\x+3/2-1/3,3*\y+3/2-1/3);
            }
        }

            %path
            \node[vertex] at (1/3, 1/3) {};
            \node[vertex] at (3/2-1/3, 3/2-1/3) {};
            \node[vertex] at (3/2-1/3, 3/2+1/3) {};
            \node[vertex] at (1/3, 3-1/3) {};
            \node[vertex] at (1/3, 3+1/3) {};
            \node[vertex] at (3/2-1/3, 3+3/2-1/3) {};
            \node[vertex] at (3/2-1/3, 3+3/2+1/3) {};
            \node[vertex] at (1/3, 6-1/3) {};

        \end{tikzpicture}
    \end{subfigure}

%% file: Img_cross.tex
\begin{tikzpicture}[
	scale=0.80000,
	back/.style={loosely dotted, thick},
	edge/.style={color=black!95!black, thick},
 	vertex/.style={inner sep=1pt,circle,draw=black!25!black,fill=blue!75!black,thick,anchor=base},
	]
	\tikzstyle{every node}=[font=\tiny]

    \foreach \x in {0,...,2} {
        \foreach \y in {0,...,2} {
            \foreach \u in {-1/3,1/3} {
                \foreach \v in {-1/3,1/3} {
                %vertices
                    \node[vertex] at (3*\x+\u,3*\y+\v) {};
                    %edges between squares
                        \draw[edge] (3*\x+\u,3*\y+\v) -- (3*\x+3*\u,3*\y+3*\v);
                    }
                }
                %edges of squares
            \draw[edge] (3*\x-1/3,3*\y-1/3) -- (3*\x+1/3,3*\y-1/3) -- (3*\x+1/3,3*\y+1/3) -- (3*\x-1/3,3*\y+1/3) -- (3*\x-1/3,3*\y-1/3);
        }
    }
    \foreach \x in {0,...,1} {
        \foreach \y in {0,...,1} {
            \foreach \u in {-1/3,1/3} {
                \foreach \v in {-1/3,1/3} {
                %vertices
                    \node[vertex] at (3*\x+3/2+\u,3*\y+3/2+\v) {};
                    %edges between squares
                        \draw[edge] (3*\x+3/2+\u,3*\y+3/2+\v) -- (3*\x+3/2+3*\u,3*\y+3/2+3*\v);
                    }
                }
                %edges of squares
            \draw[edge] (3*\x+3/2-1/3,3*\y+3/2-1/3) -- (3*\x+3/2+1/3,3*\y+3/2-1/3) -- (3*\x+3/2+1/3,3*\y+3/2+1/3) -- (3*\x+3/2-1/3,3*\y+3/2+1/3) -- (3*\x+3/2-1/3,3*\y+3/2-1/3);
        }
    }

    %labels
    \node at (0,0) [label={[shift={(0.0,-0.4)}]$(+,+)$}]{};
    \node at (3,0) [label={[shift={(0.0,-0.4)}]$(-,+)$}]{};
    \node at (6,0) [label={[shift={(0.0,-0.4)}]$(+,+)$}]{};
    \node at (0,3) [label={[shift={(0.0,-0.4)}]$(+,-)$}]{};
    \node at (3,3) [label={[shift={(0.0,-0.4)}]$(-,-)$}]{};
    \node at (6,3) [label={[shift={(0.0,-0.4)}]$(+,-)$}]{};
    \node at (0,6) [label={[shift={(0.0,-0.4)}]$(+,+)$}]{};
    \node at (3,6) [label={[shift={(0.0,-0.4)}]$(-,+)$}]{};
    \node at (6,6) [label={[shift={(0.0,-0.4)}]$(+,+)$}]{};
    
    \node at (0,1) [label={[shift={(0.0,-0.4)}]$+$}]{};
    \node at (0,4) [label={[shift={(0.0,-0.4)}]$+$}]{};
    \node at (0,7) [label={[shift={(0.0,-0.4)}]$+$}]{};
    \node at (3/2,3/2+1) [label={[shift={(0.0,-0.4)}]$+$}]{};
    \node at (3/2,6-3/2+1) [label={[shift={(0.0,-0.4)}]$+$}]{};

    \node at (3,1) [label={[shift={(0.0,-0.4)}]$-$}]{};
    \node at (3,4) [label={[shift={(0.0,-0.4)}]$-$}]{};
    \node at (3,7) [label={[shift={(0.0,-0.4)}]$-$}]{};
    \node at (3+3/2,3/2+1) [label={[shift={(0.0,-0.4)}]$-$}]{};
    \node at (3+3/2,6-3/2+1) [label={[shift={(0.0,-0.4)}]$-$}]{};

    \node at (6,1) [label={[shift={(0.0,-0.4)}]$+$}]{};
    \node at (6,4) [label={[shift={(0.0,-0.4)}]$+$}]{};
    \node at (6,7) [label={[shift={(0.0,-0.4)}]$+$}]{};
    
\end{tikzpicture} 

%% file: Img_edges.tex
\begin{tikzpicture}[
	scale=1.00000,
	back/.style={loosely dotted, thick},
	edge/.style={color=black!95!black, thick},
	edgeb/.style={color=black, thick},
	edgeo/.style={dashed, color=black, thick},
	facet/.style={fill=gray!95!black,fill opacity=0.300000},
	faceto/.style={fill=orange!95!black,fill opacity=0.300000},
	vertex/.style={inner sep=1pt,circle,draw=black!25!black,fill=blue!75!black,thick,anchor=base},
    ]
	\foreach \x in {0,...,7} {
        %vertices
        \node[vertex] (\x) at (\x+1/3,0) {};
        \node[vertex] (-\x) at (\x-1/3,0) {};
        \node[vertex] at (\x,1/3) {};
        \node[vertex] at (\x,-1/3) {};
        %squares
        \draw[edge] (\x+1/3,0) -- (\x,1/3) -- (\x-1/3,0) -- (\x,-1/3) -- (\x+1/3,0);
        %edges between squares
        \draw[edge] (\x+1/3,0) -- (\x+1/2,0);
        \draw[edge] (\x-1/3,0) -- (\x-1/2,0);
        %up and down pointing edges
        \draw[edge] (\x,1/3) -- (\x,1/2);
        \draw[edge] (\x,-1/3) -- (\x,-1/2);
    }
    %curved edges
    \foreach \x in {0,1} {
        \foreach \y in {0,1} {
            \draw[edgeb] (4*\x+\y,1/3) to [out=30, in=150] (4*\x+\y+2,1/3); 
        }
    }	
    \foreach \x in {0,...,3} {
        \draw[edgeo] (\x,1/3) to [out=30, in=150] (\x+4,1/3);    
    }	
\end{tikzpicture}	